\documentclass[english]{smfart}
\usepackage{amssymb}
\usepackage{amsmath, graphicx}
\usepackage{tikz}
\usepackage[leqno]{amsmath} 
\usepackage[T1]{fontenc}
\usepackage[latin1]{inputenc} 
\usepackage{graphpap}
\DeclareMathOperator{\Rea}{Re}
\DeclareMathOperator{\Ima}{Im}
\renewcommand{\leq}{\leqslant}
\renewcommand{\geq}{\geqslant}
\theoremstyle{plain}
\newtheorem{theo}{Theorem }
\newtheorem{prop}{Proposition}[section]

\newtheorem{coro}[prop]{Corollary}

\newtheorem*{assum}{Assumption}

\theoremstyle{definition}
\newtheorem{rem}[prop]{Remark}

 \def\pasdegrille{\let\grille =
\pasgrille}  \def\aat#1#2#3{ \divide
\dimen1 by 48 \dimen3=\dimen1 \multiply \dimen1 by #1 \advance \dimen1
by -\dimen3 \divide \dimen1 by 101 \multiply \dimen1 by 100 \divide
\dimen2 by \count11 \multiply \dimen2 by #2
\setbox0=\hbox{#3}\ht0=0pt\dp0=0pt \rlap{\kern\dimen1 \vbox
to0pt{\kern-\dimen2\box0\vss}}\dimen1= \wd1 \dimen2=\ht1}
\def\pasgrille{ \count12= \dimen1 \divide \count12 by 50 \divide
\dimen2 by \count12 \count11 =\dimen2 \ \divide \dimen1 by 48
\setlength{\unitlength}{\dimen1} \smash{\rlap{\ }} \dimen1= \wd1
\dimen2=\ht1 } \def\grille{ \count12= \dimen1 \divide \count12 by 50
\divide \dimen2 by \count12 \count11 =\dimen2 \ \divide \dimen1 by 48
\setlength{\unitlength}{\dimen1} \smash{\rlap{\graphpaper[1](0,0)(50,
\count11)}} \dimen1= \wd1 \dimen2=\ht1 }

\pasdegrille

\numberwithin{equation}{section}
\begin{document}
\title[Decay for Schr\"odinger]{ Decay of local energy for solutions of the free Schr\"odinger equation in exterior domains}

\author {N. Burq}
\address{Universit\'e Paris-Sud, Math\'ematiques, Bat 425, 91405 ORSAY CEDEX, FRANCE}
\email{Nicolas.burq@math.u-psud.fr}
\author{B. Ducomet}
\address{Universit\'e Paris-Est, LAMA (UMR 8050), UPEMLV, UPEC, CNRS, 61 Avenue du G\'en\'eral de Gaulle, 94010 CRETEIL CEDEX 10, FRANCE}
\email{bernard.ducomet@u-pec.fr}
\begin{abstract}
In this article, we study the decay of the solutions of Schr\"odinger equations in the exterior of an obstacle. 
The main situations we are interested in are the general case (no non-trapping assumptions) or some weakly trapping situations
\end{abstract}
\begin{altabstract}
On s'intéresse dans cet article à la d\'ecroissance temporelle de la solution du probl\`eme ext\'erieur pour l'\'equation de Schr\"odinger,
dans le cas d'obstacles. Les situations auquelles on s'intéresse sont le cas général (pas d'hypothèse géométrique) ou alors des situations faiblement captantes.
\end{altabstract}
\thanks{N. Burq is partially supported by ANR projects ANR-13-BS01-0010-03, (ANA\'E) and ANR-16-CE40-0013 (ISDEEC)}
%\thanks{B. Ducomet is partially supported by ANR (project~ANR-10-BLAN 0101 NoNAP)}
\maketitle
\section{Introduction}

Since the early works of Lax and Phillips \cite{LP,Mel} on exterior Dirichlet problem for the wave operator $\partial_t^2-\Delta$, it is well known
that geometry of the boundary plays a crucial role in the time decay properties of the solution whether it is trapping or not.
If the boundary is not trapping, it is known \cite{M} that local energy decays at exponential rate.
For the same problem, under very general trapping assumptions,
N. Burq  proved \cite{B} that there exists an exponentially small neighbourhood of the energy real axis, free of resonances, which implies a logarithmic 
decay rate of local energy.
\vskip0.25cm
For the Schr\"odinger operator $\imath\partial_t+\Delta$ which enjoys different spectral properties the only available result to our knowledge, concerning the exterior non-trapping
case, is due to
 Tsutsumi \cite{T1}(see also Hayashi \cite{H} for star shaped geometries and  Rauch \cite{R} for the potential case), and reads as follows.

Consider a smooth compact obstacle $\Theta \subset \mathbb{R}^d$ such that $\Omega = \mathbb{R}^d \setminus \Theta$ is connected.
\begin{equation}
\left\{  \begin{array}{ll}
{\displaystyle \bigl(\imath\partial_t+\Delta\bigr)u=0\ \ \ \mbox{in}\ {\mathbb R}_+\times \Omega},\\\\
{\displaystyle u(x,t)=0\ \ \ \mbox{on}\ {\mathbb R}_+\times \partial\Omega},\\\\
{\displaystyle u\mid_{t=0}=u_0 \in L^2 ( \Omega)\ \ \ \mbox{in}\ \Omega,}
 \end{array}
\right.
\label{schrod}
\end{equation}

 Denoting $L^2_R(\Omega):=\{u\in L^2(\Omega),\ \mbox{Supp}\ u\subset \Omega\cap B(0,R)\}$, one get~\cite{T1}
\begin{theo}[Tsutsumi]
\label{t0}
Let $d\geq 3$ and $\Omega$ a non-trapping exterior domain  (i.e. such that any singularity of the Green function for the  Dirichlet problem for the wave equation in $\Omega$ goes to infinity
as $t\rightarrow \infty$). Let $U(t)$ be the evolution operator associated to problem (\ref{schrod}), and $R_1,R_2>0$. There exists $C>0$ such that for all $t>1$
\begin{equation}
\left\| U(t) \right\|_{L^2_{R_1}(\Omega),L^2_{R_2}(\Omega)}
\leq \frac{C}{ t^{\frac{d}{2}}}.
\label{e0}
\end{equation}
\end{theo}
 
 In this article our purpose is to  generalize the result of Tsutsumi in the general trapping case by adapting the strategy  of \cite{B} to a generalized Schr\"odinger equation, with a Laplace-Beltrami
 operator 
 $$ \frac{1}{\varrho(x)}\ \mbox{div} A(x)\nabla,$$
  replacing the Laplacian $\Delta$, and the results  of Ikawa \cite{Ik1, Ik2} in the case where $\overline{\mathcal O}$ is the union 
of several disjoint convex obstacles. 

An extra motivation for this work is the relevance of
this kind of information for global existence of solutions to related non-linear Schr\"odinger's equations
 (see \cite{T2} \cite{T3} \cite{BGT}).
\vskip0.5cm 
 {\bf Acknowledgements. } We would like to thank C. Zuily for various discussions on the content of this article and a referee for signaling several imprecisions in a first version of the paper and raising comments which lead to significant improvements (see~\eqref{estibis} in Theorem~\ref{t4}).

\section{Hypotheses and results}

Let us consider a compact set $\Theta\subset {\mathbb R}^d$ with $C^{\infty}$ boundary such that $\Theta\subset B(0,a)$ for a $a>0$.
We denote by $P$ the operator defined by
\begin{equation}
{\displaystyle P(x,\partial):=-\frac{1}{\varrho(x)}\ \sum_{i,j=1}^n \partial_i a_{ij}(x)\partial_j},
\label{op}
\end{equation} 
in $\Omega:=\Theta^c$
 associated to the weight $\varrho(x)\in C^{\infty}(\overline{\Omega})$ such that $\varrho(x)\geq \varrho_0>0$
and the symmetric and positive definite matrix $A$ with entries
$A:=(a_{ij})\in C^{\infty}(\overline{\Omega})$ satisfying the ellipticity condition
\begin{equation}
{\displaystyle \sum_{i,j=1}^n a_{i,j}(x)\xi_i\xi_j\geq C|\xi|^2,\ \ \ \forall(x,\xi)\in T^*\Omega,\ \ C>0}.
\label{ellip}
\end{equation}
We suppose that (``asymptotic flatness") for any multi-index $\alpha$ such that $|\alpha|\leq 1$, there exist $\beta>0$ and $M_{\alpha}>0$ 
such that $\forall x\in {\mathbb R}^d\backslash B(0,a)$
\begin{equation}
{\displaystyle \left| \partial^{\alpha}_x\bigl(\varrho(x)-1\bigr)\right|+\sum_{i,j=1}^n \left|\partial^{\alpha}_x \bigl(a_{i,j}(x)-\delta_{i,j}\bigr)\right|
\leq \frac{M_{\alpha}}{ \langle x\rangle^{\beta+|\alpha|}},}
\label{bornes}
\end{equation}
where we denote by $ \langle \cdot\rangle$ the usual "japanese bracket"
\begin{equation}
 \langle x\rangle= \bigl(1+|x|^2\bigr)^{1/2}.
 \end{equation}
In the following we will also denote by $P$ the self-adjoint extension of (\ref{op}) in ${\mathcal H}=L^2(\Omega,\varrho(x)\ dx)$
with domain $D(P)=\{u\in H^2(\Omega),\ \ \left.  u\right|_{\partial\Omega}=0\}$. We shall denote by $\mathcal{H}^k = D( P^{k/2})$ endowed with the norm $\|u\|_{\mathcal{H}^k}^2 = \| P^{k/2} u\|_{L^2}+ \| u\|_{L^2}^2$, and by $\mathcal{H}^{-k}$ the image of $L^2$ by the map 
$$ u \mapsto (P+ \text{ Id} ) u, $$ endowed with the norm 
$$ \| u\|_{\mathcal{H}^{-k}} = \| (P+ \text{ Id} )^{-1} u \|_{L^2}.$$

If we denote by $\widetilde{R}(z)$ the outgoing resolvent $(z^2-P)^{-1}$ of $P$ in $\Omega$, $\bigl(z-P\bigr)^{-1}u_0$ is the solution of the problem
\begin{equation}
\left\{  \begin{aligned}&\bigl(z^2-P\bigr)v=u_0\ \ \text{ in }  \Omega,\\
&v \mid_{\partial \Omega}=0,\\
& \left(\frac{x}{|x|}\cdot \nabla_x - i  \sqrt{z} \right)v = o_{ |x| \rightarrow + \infty }( |x| ^{(1-n)/2}), \text{ ( $v$  is outgoing) }
 \end{aligned}
\right.
\label{stat}
\end{equation}
 and we have
\[
\widetilde{R}(z)f=i\int_0^{\infty} e^{-it\sqrt{z}}u(t)\ dt,
\]
where $u$ is the solution de (\ref{schrod}) and the square root is defined on $\mathbb{C}\setminus \mathbb{R}^+$ by 
$$ \sqrt{ \rho e^{i \theta} }= \sqrt{ \rho} e^{i \theta/2}, \rho >0, \theta \in ( - 2 \pi, 0).$$

It is known \cite{Mel} that $z\rightarrow \widetilde{R}(z)$
extends from $\Ima\ z<0$ (the physical half plane) to a region ${\mathcal O}$ (if $d$ is odd, ${\mathcal O}$ is the complex plane ${\mathbb C}$, 
and if $d$ is even, ${\mathcal O}$ is the simply connected covering of ${\mathbb C}^*$) 
as a meromorphic operator from $L^2_{comp}(\Omega)$ into $H^2_{loc}(\Omega)$. Then the resolvent $R(z) = \mathcal{R}( \sqrt{z}) $
 extends meromorphically from $\mathbb{C}^*$ to its universal cover (even dimensions) and to the two leaves cover (odd dimensions). 

Let us recall the general result~\cite{B, B2}
\begin{theo}[Burq]
\label{p2}
Assume that the operator $P$ satisfies the assumption (AE) in Remark~\ref{rem.ert}.

There exist $a, D>0$ such that, for any $\chi \in C^\infty_0( \mathbb{R}^d)$, in the conditions of Theorem \ref{t1}, the operator $\chi R(z)\chi $ 
has an analytic continuation, as a continuous operator in ${\mathcal H}$, 
in the region 
\begin{equation}
\Omega=\left\{ z \in \mathbb{C}; \Ima z \leq  e^{-a \sqrt{|\Rea z|}}:|z|\geq D\right\}.
\label{e5}
\end{equation}
Moreover, in this region, there exist $A>0$ and $C>0$ such that 
\begin{equation}
\left\|\chi R(z)\chi \right\|_{{\mathcal L}({\mathcal H})}
\leq
C e^{A\sqrt{|\Rea z|}}.
\label{e6}
\end{equation}
\end{theo} 
From these  general estimates, following Vodev, we deduce  resolvent estimates with space weights (see~\cite[Corollary 1.2]{Vo2}) 
\begin{theo}[Vodev]
\label{t1}
There exists positive constants $C_1,C_2$ and $\epsilon_0$ such that $R(z)$, which is holomorphic in $\Ima \xi <0$, does not have any pole in the region
\begin{equation}
{\mathcal V}:=\left\{z\in {\mathbb C}\ :\ |\Ima  z|<\epsilon_0 e^{-C_2\sqrt{|\Rea z|}}\right\}\cap \{|z|>C_1\},
\label{e1}
\end{equation}
where $(1+i0)^{1/2}>0$.

Moreover for any $s>1/2$ and for any $z\in {\mathcal V}$ the resolvent $R$ satisfies the bound
\begin{equation}
{\displaystyle \| e^{-\langle x\rangle} R(z)e^{-\langle x\rangle}\|_{{\mathcal L}({\mathcal H},{\mathcal H})}
\leq C \ e^{C_3\sqrt{|\Rea z|}}.}
\label{e2}
\end{equation} 
\end{theo}
This result will imply as in \cite{B}, a logarithmic decay for the solution of the problem 
(\ref{schrod}).

 More precisely our first result reads 
\begin{theo}
\label{t2}
For all $R_1,R_2>0$ and any $k>0$, there exists $C>0$ such that for any initial condition $u_0\in D(P^k)$ supported in
$B(0,R_1)\cap \Omega$, one has for any $t>1$
\begin{equation}
\left( \int_{B(0,R_1)\cap \Omega} |u(x,t)|^2 dx\right)^{1/2}
\leq \frac{C}{(\log t)^{2k}}\ \|u_0\|_{\mathcal{H}^{k}}.
\label{e3}
\end{equation}
\end{theo}

Our second result concerns some weakly trapping situations where resolvent estimates stronger than~\eqref{e2} are available.
 For the wave equation, such situations were investigated by  Ikawa in a series of articles 
(see \cite{Ik1},  \cite{Ik2},  \cite{Ik3}, \cite{Ik4}),  G\'erard \cite{Ge} and more recently by Petkov-Stoyanov \cite{PS}.

 We now assume that the operator $P$ satisfies
\begin{itemize}
\item (E) The metric in Section~2 is exactly euclidean outside of a fixed ball. 
%\item(AE) The metric in Section~2 satisfies near infinity the dilation analytic asymptotic flatness of the previous page 
\item{($H_N$)} The weighted resolvent  $\chi(x) R(z)\chi(x)$ which is holomorphic in $\Ima z <0$ admits a holomorphic extension in a region 
\begin{equation}
\label{domaine}
\{z\in{\mathbb C}; |\Rea z| \geq C, \Ima z \leq c|z|^{-N}<C\},
\end{equation}
 and satisfies for some $s>\frac 1 2$ and some compactly supported function with large enough support, $\chi$,
\begin{equation}
{\displaystyle \| \chi(x) R(z) \chi(x)\|_{{\mathcal L}({\mathcal H},{\mathcal H})}
\leq C |z|^N.}
\label{e2bis}
\end{equation} 
\end{itemize}
\begin{rem}
In the litterature, motivated by wave equation expansion, authors usually study the resolvent 
$$\widetilde{R}(z) = R(z^2)= ( -\Delta - z^2)^{-1} .$$
A straightforward calculation shows that stating Assumption $H_N$ in this context amounts simply to replace the domain~\eqref{domaine} by 
$$ \{z\in{\mathbb C}; |\Rea z| \geq C, \Ima z \leq c|z|^{-N-\frac 1 2}\},$$ and the estimate~\eqref{e2bis} by 
$$
{\displaystyle \| \chi(x) R(z)\chi(x)\|_{{\mathcal L}({\mathcal H},{\mathcal H})}
\leq C |z|^{2N}.}
$$
As a consequence, $H_{- \frac 1 2}$ is actually equivalent to a non trapping assumption while for larger $N$, it corresponds to a weakly trapping one. 
\end{rem}

\begin{theo}
\label{t4}
Assume that  (E) and ($H_N$) % or (AE) and ($H_N$) 
are satisfied.

 Then 
\begin{itemize}
\item If $0<N$, for any real $k>0$,   there exists  $C_k>0$ 
such that  for any  $t\geq 1$,we have 
\begin{equation}
\left\| \langle x\rangle^{-d/2- \epsilon}\ \frac{e^{-itP}}{(1-iP)^k}\ \langle x\rangle^{-d/2-\epsilon}\right\|_{{\mathcal L}({\mathcal H})}
\leq C_k
\max  \Bigl(\frac 1{t^{\frac d 2 }} , \Bigl(\frac{\log(t)}{t} \Bigr) ^{\frac k N}\Bigr) \label{esti}
\end{equation}
\item If $-\frac 1 2 \leq N<0$, for any real $0\leq k< d/2- (1+2N)$,  and any $\epsilon >0$, there exists  $C_{k, \epsilon}>0$ 
such that  for any  $t\geq 1$,we have 
\begin{equation}
\label{estibis}
\left\| \langle x\rangle^{-d/2}\ e^{-itP}\  (1-iP)^k \langle x\rangle^{-d/2}\right\|_{{\mathcal L}({\mathcal H})}
\leq C_k
 \frac 1{t^{\frac d 2 - \epsilon }}.
\end{equation}
\end{itemize}
\end{theo}
\begin{rem} Notice that as noticed by a referee, since the operator $\langle x \rangle ^{ d/2- \epsilon} ( (1- i P)^{-k} \langle x \rangle ^{ -d/2- \epsilon}$ is bounded on $L^2$,  ~\eqref{estibis} also holds for $k \leq 0$.
\end{rem}
\begin{rem}
\label{r1} In the case of two strictly convex obstacles $\Theta= {\mathcal O}_1\cup {\mathcal O}_2$ in ${\mathbb R}^d$
(or of several obstacles satisfying a pressure condition),
  relying on the works by Ikawa \cite{Ik1, Ik2, Ik3, Ik4}, Burq~\cite{B} showed that the 
assumptions $(H_N)$ above are satisfied for any $N> - \frac 1 2$ (actually the loss with respect to the non trapping
 estimates, $N= - \frac 1 2$ is only logarithmic). See also~\cite{Ch, NZ}. In the case of degenerate obstacles the work of Ikawa~\cite{Ik5, Ik6} suggests that Assumption ($H_N$) should be in this case true for some $N>0$ (depending on the order of vanishing of the curvature). It is satisfied for some semi-classical models (see~\cite[Section 6.4]{BuZw}).
\end{rem}
\begin{rem} The self adjointness of $P$ could probably (see~\cite{B}) be replaced by the assumption that $iP$  is maximal dissipative. 
For simplicity, we do not pursue in this direction here.
\end{rem}
\begin{rem}\label{rem.ert}
Another natural assumption would be 
\begin{assum}[AE]
The operator  $P$ satisfies (AE) if and only if it satisfies ~\eqref{op} \eqref{ellip}  and there exists $R_0>0, \theta_0>0, \epsilon _0 >0$ 
such that its coefficients are analytic in the domain
$$ \Lambda_{\theta_0, \epsilon_0, R_0}
 = \{ z = r \omega \in \mathbb{C}^d; 
\omega \in \mathbb{C}^{d}\text{ dist} ( \omega; \mathbb{S}^{d-1})< \epsilon _0, r \in \mathbb{C}, |r| \geq R_0, \text{ arg} (r) \in [ - \theta_0, \theta_0]\},   $$
(where $\mathbb{S}^{d-1}$ is the real sphere) and satisfies ~\eqref{bornes} in $\Lambda_{\theta_0, \epsilon_0, R_0} $.
\end{assum}
We are not aware of precise low frequency resolvent estimates under assumption (AE) for boundary value problems, though it is likely that they should hold. Provided such estimates were true, we could generalize our results to  this more general setting
\end{rem} 
\section{Gluying estimates: from space-local results to space-weighted estimates}

In this section we prove that it is enough to prove the results in Theorem~\ref{t4} with the weight $\langle x \rangle ^{-1} $ replaced by a cut off $\chi(x)$, $\chi \in C^\infty_0 ( \mathbb{R}^d)$. 
 In the remaining of this section, we assume that Theorem~\ref{t4} is true with $\langle x \rangle ^{-1} $replaced by a cut--off $\chi\in C^\infty_0 ( \mathbb{R}^d)$. 

Namely
\begin{itemize}
\item If $0<N$, for any real $k>0$,   there exists  $C_k>0$ 
such that  for any  $t\geq 1$,we have 
\begin{equation}
\left\| \chi(x)\frac{e^{-itP}}{(1-iP)^{k/2}}\ \chi(x)\right\|_{{\mathcal L}({\mathcal H})}
\leq C_k
%\begin{cases}
\max  \Bigl(\frac 1{t^{\frac d 2 - \epsilon}} , \Bigl(\frac{\log(t)}{t} \Bigr) ^{\frac k N- \epsilon}\Bigr) \text{ if } (E) \text { and } H_N \text {\ are satisfied}.\\
%max  \Bigl( \frac 1{t^{\frac d 2 - \epsilon }} , \Bigl(\frac{\log(t)}{t} \Bigr) ^{\frac k N}\Bigr)\text{ if } (AE) \text { and } H_N \text {\ are satisfied}
%\end{cases}
\label{estiter}
\end{equation}
\item If $-\frac 1 2 \leq N<0$, for any real $0\leq k< d/2- (1+2N)$,  and any $\epsilon >0$, there exists  $C_{k, \epsilon}>0$ 
such that  for any  $t\geq 1$,we have 
\begin{equation}
\label{estiquar}
\left\| \chi(x)\ e^{-itP}\  (1-iP)^{k/2} \chi(x)\right\|_{{\mathcal L}({\mathcal H})}
\leq C_k
 \langle t \rangle^{-\frac d 2 + \epsilon }.
\end{equation}
\end{itemize}
The idea is to use,  as a black box, some results by Bony and H\"affner that we recall below (see~\cite[Theorems 1-iii), 5-i)]{BoHa}
\begin{theo}[Bony-H\"afner]
Assume that there is no obstacle and that the operator $P_0$ satisfies assumptions~\eqref{op}, \eqref{ellip} and furthermore is {\em non trapping}.

 Then we have 
\begin{equation}\label{boha1}
\exists C>0; \forall |t| \geq 1, \|\langle x\rangle ^{- d/2} e^{it P_0} \langle x\rangle ^{- d/2} \|_{\mathcal{L} ( H^{- d /2}; L^2)}\leq C\langle t\rangle ^{- d/2+ \epsilon}.
\end{equation} 
Furthermore, for all $\phi \in C^\infty_0 (]0, + \infty[)$ and $\mu \geq 0$ there exists $h_0>0$ and $C>0$ such that for all $0<h<h_0$, and all $t\neq 0$,
\begin{equation}\label{boha2}
\|\langle x\rangle ^{- \mu} e^{it P_0} \phi( h^2 P_0)\langle x\rangle ^{- \mu} \|_{\mathcal{L} ( L^2)}\leq \langle th^{-1}\rangle ^{-\mu}.
\end{equation} 
\end{theo}
\begin{coro}\label{cor.3.1}
Assume that there is no obstacle and that the operator $P_0$ satisfies assumptions~\eqref{op}, \eqref{ellip} and furthermore is {\em non trapping}. Then for any $0< \nu < \mu$, $\nu <1$, $k\in \mathbb{R}$, there exists $C>0$,
\begin{equation}\label{boha3}
\|\langle x\rangle ^{- \mu} e^{it P_0} \langle x\rangle ^{- \mu} \|_{\mathcal{L} ( H^k( \mathbb{R}^d); L^1(-1,1); H^{k+\nu}( \mathbb{R}^d)))}\leq C.
\end{equation} 
\end{coro}
\begin{proof} For simplicity we just prove the result for $k=0$, the general case beeing similar. For fixed $h>0$, 
$$ \| \langle x\rangle ^{- \mu} e^{it P_0} \phi( h^2 P_0)\langle x\rangle ^{- \mu} \| _{H^\nu( \mathbb{R}^d)} \sim h^{-\nu} \|  \| \langle x\rangle ^{- \mu} e^{it P_0} \phi( h^2 P_0)\langle x\rangle ^{- \mu} \| _{L^2( \mathbb{R}^d)},$$
and consequently according to~\eqref{boha2}
\begin{multline}
 \| \langle x\rangle ^{- \mu} e^{it P_0} \phi( h^2 P_0)\langle x\rangle ^{- \mu} \|_{\mathcal{L} ( L^2( \mathbb{R}^d); L^1(-1,1); H^\nu( \mathbb{R}^d))}\\
 \leq C h^{- \nu} \int_0^1 \langle th^{-1} \rangle ^{- \mu} dt \leq C h^{1- \nu} \max (1, h^{\mu -1}) \leq C h^ \delta,
 \end{multline}
 where $\delta >0$ and the estimate~\eqref{boha3} follows from the dyadic decomposition
 $$ u = \Psi(P_0) u + \sum_{j=1}^{+\infty} \phi (2^{-2j} P_0) u_0,$$
  and  summation.
 \end{proof}

 We now turn to proving the same type of estimate under assumption $H_N$ in the presence of an obstacle
 \begin{prop}\label{prop.3.2} Assume that the operator satisfies (E) and ($H_N$).  Then 
 for any $0< \nu < \mu$, $\nu <-2N$,  there exists $C>0$,
\begin{equation}\label{boha3}
\|\langle x\rangle ^{- \mu} e^{it P_0} \phi( h^2 P_0)\langle x\rangle ^{- \mu} \|_{\mathcal{L} ( H^k( \mathbb{R}^d); L^1(-1,1); H^{k+\nu}( \mathbb{R}^d)))}\leq C.
\end{equation} 
 \end{prop} 
 \begin{proof} 
 Again, for simplicity we just prove the result for $k=0$. Let us first prove this estimate with the weight $\langle x \rangle^{- \mu}$ replaced by a cut off $\chi (x)$. We start from the formula 
$$\chi (x)  e^{itP} \phi( h^2 P) \chi(x) = \int _{\Im z = 0^- } e^{itz } \phi(h^2 z) \chi(x) ( P+z)^{-1} \chi(x) dz,$$
and perform the change of variables 
$\zeta = h^{-2N} z$, which gives 
$$\chi (x)  e^{itP} \phi( h^2 P) \chi(x) = h^{2N} \int _{\Im \zeta  = 0^- } e^{ih^{2N}t\zeta  } \phi(h^{2+ 2N} \zeta)) \chi(x) ( P+h^{2N} \zeta) )^{-1} \chi(x) dz,$$
and according to~\eqref{domaine} the r.h.s. is holomorphic in the region
$ \Im \zeta <c$. We use the quasi-holomorphic extension of the function $\phi_h( \zeta) = \phi(h^{2+ 2N} \zeta)$, $\phi_h^k$ defined by 
$$  \phi_h^k (x+ iy)= \sum_{j=0}^k \frac{  d^k} {dx^k } \phi_h (x) \frac { ( iy)^j }{ j !} \rho(y),$$
with $\rho \in C^\infty_0 (0,c)$ equal to $1$ near $0$, 
which satisfies 
$$ |\partial_{\overline{z} }  \phi_h^k ( x+ iy)) | \leq C |\rho'(y)|  + C h^{2(1+N) (k+1)}  |y|^{k} | \rho(y)|.$$
As a consequence, using Green's formula and~\eqref{e2bis}, we obtain (using that the support of $\phi_h$ is included in the set $|\zeta| \leq h^{-2-2N}$)
 \begin{multline}\| \chi (x)  e^{itP} \phi( h^2 P) \chi(x)\|_{\mathcal{L}(L^2)}  = h^{2N} \left\| \int _{ \zeta = x+ iy, y >0 } e^{ih^{2N}t\zeta  } \partial_{\overline{\zeta}} \phi_h^k (\zeta) \chi(x) ( P+h^{2N} \zeta )^{-1} \chi(x) |d\zeta| \right\| ,\\
 \leq C h^{2N}h^{-2- 2N}h^{-2N} \int_{y=0}^c  e^{-t h^{2N }y}( |\rho'(y)|  + C h^{2(1+N) (k+1)}  |y|^{k} ) dy 
 \\
 \leq C h^{-2- 2N} \bigl( e^{-t h^{2N} }+ C_k (\frac{ h^{2}} {t} )^{k+1}\bigr).
 \end{multline}
 And consequently, using this estimate for $|t| \geq h^{-2N- \epsilon}$ and the trivial bound (by $1$) otherwise, we get 
  
 $$ \| \chi (x)  e^{itP} \phi( h^2 P) \chi(x)u_0 \|_{L^1(-1,1); H^\nu}\leq h^{- \nu} ( h^{-2N - \epsilon} + O(h^\infty)),$$
 and the result follows if $\epsilon >0$ is small enough.
 Now we extend the result by replacing the $\chi(x)$ weight on the left by $\langle x \rangle ^{- \mu}$.  
Let 
$$u(t) = (1- \chi(x) )e^{itP}  \chi(x)u_0 ,$$

solution of 
\begin{equation}\label{nonhomogbis}
 (i \partial_t + P) u = - [P, \chi] e^{itP}  \chi(x)u_0.
 \end{equation}
Since $P$ is euclidean or asymptotically flat, if $\chi =1$ on a sufficiently large domain, and since $v$ and $[P,\chi] u$ are supported
 on the region where $\chi \neq 1$, we can replace $P$ in~\eqref{nonhomogbis} by an operator $P_0$ on $L^2( \mathbb{R}^d)$ (no obstacle $\Theta$),
 whose coefficients coincide with the coefficients of $P$ on the support of $1- \chi$ and which is non trapping (because, $P_0$ can be taken
 arbitrarily close to the euclidean Laplacian in $C^\infty$-norm). It consequently satisfies~\eqref{boha1} and~\eqref{boha2}, and we now have 
 
$$ (i \partial_t + P_0) u = - [P_0, \chi] e^{itP}  \chi(x)u_0 \Rightarrow (1- \chi) u = e^{itP_0 } (1- \chi) u_0+ i \int_0 ^t e^{i(t-s)} P_0   [P_0, \chi] e^{is P} \chi u_0 ds.
$$ 

The contribution of the first term is easily dealt with. To estimate the contribution of the second term, we choose $\nu_1 < - 2N$ and $\nu_2 <1$ such that $\nu = \nu_1 + \nu_2 -1$  and $\widetilde{\chi}$ equal to $1$ near the support of $\nabla \chi$, and write
\begin{multline}
\left\| \langle x \rangle ^{- \mu} \int_0 ^t e^{i(t-s) P_0}   [P_0, \chi] e^{is P} \chi u_0 ds\right\|_{L^1((-1,1)_t; H^{\nu_1+ \nu_2 -1} )} \\
\leq \int_{(t,s) \in (-1,1)^2} \left\|\langle x \rangle ^{- \mu} e^{i(t-s) P_0}\widetilde{\chi}  [P_0, \chi] \widetilde{\chi}e^{is P} \chi u_0\right\|_{H^\nu} ds dt\\
\leq \int_{(t,s) \in (-1,1)^2} \left\|\langle x \rangle ^{- \mu} e^{i(t-s) P_0} \widetilde{\chi} \right\|_{H^{\nu_1-1} \rightarrow H^{\nu_1 + \nu_2 -1} }  \left\| \widetilde{\chi} e^{is P} \chi\right\| _{L^2 \rightarrow H^{\nu_1}} ds dt \ \| u_0 \|_{L^2} 
\leq C \|u_0 \|_{L^2}. 
\end{multline}
This shows that we can replace  the $\chi(x)$ weight on the left by $\langle x \rangle ^{- \mu}$. To replace the  $\chi(x)$ weight on the right by $\langle x \rangle ^{- \mu}$, we argue by duality and proceed in a similar way: here we have to deal with operators norms of the adjoint operator 
$$ \langle x \rangle ^{- \mu} \int_{-1}^1 e^{-isP_0} \langle x \rangle ^{- \mu}  f(s) ds,$$ and bound its norm from $L^\infty H^{-1} $ to $L^2$, knowing that the similar bound holds when the weight on the left is replaced by a cut-off. We  express $ (1- \chi) e^{-isP_0} $ in terms of $ e^{-isP}$ 
and the proof follows by similar considerations as previously. 
\end{proof}
Let us now come back to the proof of Theorem~\ref{t4} knowing that ~\eqref{estiter} and~\eqref{estiquar} hold.  
 As a first step, we show that we can replace the cut off $\chi$ on the left in~ \eqref{estiquar} by the weight $\langle x \rangle ^{- d/2}$.
 We write again
$$ \langle x \rangle ^{- d/2} (1- \chi) (1-iP)^k e^{-it P} \chi  = (1-iP)^k    (1-iP)^{-k} \langle x \rangle ^{- d/2} (1- \chi) (1-iP)^k  \langle x \rangle ^{ d/2}  \langle x \rangle ^{- d/2} e^{-it P} \chi,$$
and by standard pseudodifferential calculus (away from the boundary) we have that the operator  
$$  A=  (1-iP)^{-k} \langle x \rangle ^{- d/2} (1- \chi) (1-iP)^k  \langle x \rangle ^{ d/2}, $$ is, modulo smoothing and decaying terms, equal to 
$A (1- {\chi_1}) $ (if $\chi $ is equal to $1$ on the support of $ {\chi_1}$) and is bounded on $H^k$. As a consequence, it is enough to estimate 
 $  \langle x \rangle ^{- d/2} (1- {\chi_1}) e^{-it P} \chi$ from $L^2$ to $H^k$. 
   Let $u = e^{itP} \chi u_0$.
The function $v = (1- {\chi_1})\chi u$ satisfies 
\begin{equation}
\label{nonhomog}
( i \partial_t  - P ) v = -[ P, {\chi_1}] u .
\end{equation}
We can, as before, replace here $P$ by $P_0$ which satisfies~\eqref{boha1}, \eqref{boha2} and according to~\eqref{nonhomog} and  Duhamel's formula, we have 
$$v= e^{-itP_0}  \chi u_0 + i \int_{s=0}^t e^{-i (t-s) P_0}  [ P, {\chi_1}] u (s) ds.
$$ 
The first term is estimated by applying~\eqref{boha1}. The idea to estimate the non homogeneous term  is to use for $e^{-i (t-s) P_0}$
 the estimate~\eqref{boha1} and for $  [ P, \chi] u (s)$ the estimate~\eqref{estiquar}.

 An issue here is the loss of one derivative due
 to the fact that the operator $P$ is of order $2$ (and hence the bracket is of order $1$). To gain back this loss of one derivative, we have to estimate carefully the contributions of the regions $|t-s| \leq 1/2$ and $|s|\leq 1/2$.  Let us first estimate the (easy) part when  both $|s|$, and $|t-s|$  are larger than $1/2$. In this case the loss of derivatives can be recovered by using the slack of $\frac d 2 \geq 1$ derivatives in~\eqref{boha1}, while the $k$ derivatives from the $H^k$ norm are estimated using~\eqref{estiter}, \eqref{estiquar} (according whether $k> 0$ or $k<0$), giving 
\begin{equation}
\label{estim} \| \langle x \rangle ^{- \frac d 2 } v \| _{H^k} \leq C \int_{s=0}^t \frac {1} {1+ |t-s|^{\frac d 2- \epsilon}} \frac 1 {1+ |s| ^{\frac d 2- \epsilon}}\| u_0 \|_{L^2}
\leq  \frac{C} {1+ |t|^{\frac d 2 - 2\epsilon}} \| u_0 \|_{L^2}.
\end{equation}

 Let us now focus on the contribution of $|t-s|\leq 1/2$. In this case we have $|s| \sim |t|$ and we bound the non homogeneous part by 
  \begin{multline}
  \int_{|t-s| < 1/2} \| \langle x\rangle^{-\frac d 2} e^{-i (t-s) P_0} { \chi} \left\|_{\mathcal{ L} ( H^{k -1 + \epsilon} ; H^k)} \right\| [P, \chi_1] \|_{\mathcal{ L} ( H^{k  + \epsilon} ; H^{k-1 + \epsilon})} \left\| \chi e^{-is P} \chi \right\|_{\mathcal{ L} (L^2;  H^{k + \epsilon}) } ds
 \\
 \leq C  \left\| \langle x\rangle^{-\frac d 2} e^{-i t P_0} { \chi} \right\|_{\mathcal{ L} ( H^{k -1 + \epsilon} ; L_1( 0,1); H^k)} 
  \sup_{|s-t|< 1/2} \left\| \chi e^{-is P} \chi \right\|_{\mathcal{ L} (L^2;  H^{k + \epsilon}) } 
  \end{multline} 
  and we use Corollary~\ref{cor.3.1} and ~\eqref{estiter} or~\eqref{estiquar} to conclude.
  
Let us now deal with the contribution of the region $|s| \leq 1/2$.  We proceed similarly and get a bound (with $k+ \epsilon < d/2$)
 \begin{multline}
  \int_{|s| < 1/2} \left\| \langle x\rangle^{-\frac d 2} e^{-i (t-s) P_0} { \chi} \right\|_{\mathcal{ L} ( H^{ - \epsilon };  H^{k})} \| [P, \chi_1] \|_{\mathcal{ L} ( H^{1-\epsilon} ; H^{ - \epsilon})} \left\| \chi e^{-is P} \chi \right\|_{\mathcal{ L} (L^2;  H^{1- \epsilon}) } ds
 \\
 \leq C  \sup_ {|\sigma -t|<1/2}\left\| \langle x\rangle^{-\frac d 2} e^{-i \sigma P_0} { \chi} \right\|_{\mathcal{ L} ( H^{- \epsilon };  H^{k})}
 \left\| \chi e^{-is P} \chi \right\|_{\mathcal{ L} (L^2;  L^1(0,1/2); H^{1-\epsilon}) }, 
  \end{multline} 
 and we use~\eqref{boha1} to estimate the first term in the r.h.s  and Proposition~\ref{prop.3.2} for the second term.

As a consequence, we just proved that in~\eqref{estiter}, \eqref{estiquar} we can put in the left $\langle x \rangle^{-\frac d 2 }$ instead of $\chi(x)$.
 By duality we hence also have a similar estimate with $\langle x \rangle^{-\frac d 2 }$ instead of $\chi(x)$ in the right (and now $\chi(x) $ in the left).
 We now re-perform the same procedure which allows to replace again $\chi(x) $ in the left by $\langle x \rangle^{-\frac d 2}$. This ends the proof of the insertion of the proper weight in ~\eqref{estiter},~\eqref{estiquar}.

\section{Proofs of theorem \ref{t2}}\label{sec.3}

\subsection{Proof of theorem \ref{t2}}
Let $\chi\in C^\infty_0$ equal to $1$ on $B(0, R_1)$. Theorem~\ref{t2} amounts to proving that for any $k>0$ there exists $C>0$ such that for any $u_0$ satisfying $u_0 = \chi u_0$, one has 
$$\| \chi e^{itP} u_0\| \leq \frac{C} { \log(t)^{2k} }\|( P+i)^{k/2} u_0\|_{L^2}.$$
Since the $L^2$ norm is conserved by the evolution $e^{itP}$, it is by interpolation clearly enough to prove the result for even integers $k$.  In that case, if $\widetilde{\chi}$ is equal to $1$ on the support of $\chi$, we have 
$$v_0 = ( P+i)^{k/2} u_0= ( P+i)^{k/2} \chi u_0 = \widetilde{\chi} ( P+i)^{k/2} \chi u_0 = \widetilde{\chi}  v_0$$
and Theorem ~\ref{t2} amounts to proving 
\begin{theo}
\label{t3}
For any even integer $k>1$ there exists a constant $C_k>0$ such that for any $t>1$
\begin{equation}
\left\|\chi(x)\ \frac{e^{-itP}}{(1-iP)^k} \chi(x)\right\|_{{\mathcal L}({\mathcal H})}
\leq \frac{C_k}{(\log t)^{2k}}.
\label{e9}
\end{equation}
\end{theo}
 \subsection{Proof of Theorem \ref{t3}}\label{se.3.3}

Let us notice first that an application of ~\cite[Th\'eor\`eme 3]{B}, using only  the weaker estimate 
\begin{equation}
{\displaystyle \|\chi(x) R(z)\chi(x)\|_{{\mathcal L}({\mathcal H},{\mathcal H})}
\leq C \ e^{C_3|\Rea z|},}
\end{equation} 
in the domain
\begin{equation}
\widetilde{{\mathcal V}}:=\left\{z\in {\mathbb C}\ :\ |\Ima z|<\epsilon_0 e^{-C_2|\Rea z|}\right\}\cap \{|z|>C_1\},
\label{e1bis}
\end{equation}
  implies directly (at least in the high frequency regime), see~\eqref{e9}
\begin{equation}
\left\|\chi(x)\ \frac{e^{-itP}}{(1-iP)^k}\chi(x)\right\|_{{\mathcal L}({\mathcal H})}
\leq \frac{C_k}{(\log t)^{k}}.
\label{e9bis}
\end{equation}

To get the proper decay rate  ($\frac{C_k}{(\log t)^{2k}}$), we shall use as in~\cite{B} a time cut-off, together with a
 complex deformation of the contour in the $\xi=z^2$-variable which is adapted to  Schr\"odinger operators. 
 
 In the rest of this section we shall prove Theorem~\ref{t3}.  We first deal with the low frequencies. According to~\cite[Section B.2]{B} we have 
 \begin{theo} 
 For any $\theta \in [0, \frac \pi 4[$ there exist $\delta >0$ and $r, C>0$ such that the resolvent 
 $R(z)$ admits an analytic expansion in the sector 
 $$\{ z= \rho e^{i \phi}\in \mathbb{C}; 0< \rho <r, \phi \in [ - \pi - \theta, \theta]\},$$ 
 and is uniformly bounded in this sector. 
 \end{theo}
Now we start with the classical formula 
\begin{equation}
\label{5.1bis}
e^{-itP}\frac{\chi (P)}{(1-iP)^k}x= \frac{1}{2i\pi}\int _{ \xi =-\infty-i0} ^{+\infty -i0}e ^{it\xi}
\frac{\chi ( \xi)}{\bigl(1-i\xi\bigr)^ k} \frac 1 {\xi +P}d\xi\ x.
\end{equation}
Let $\widetilde{\chi} $ be a quasi-analytic extension of $\chi$ i.e. a function satisfying 
$$  |\partial_{\overline{z}}  \chi (z)| \leq C_N |\Im z| ^N, \qquad \forall N, $$
such that  
$$ \Omega= \text{supp} (  \partial_{\overline{z}}  \chi ) \subset \left\{ x+ iy; |x| \in [\frac r 4 , \frac r 2] , 0\leq y \leq  \frac C 8\right\}.$$

By contour deformation on the contour $\Gamma$ of Figure~\ref{fig.0},
 denoting by $\Gamma_-$ the region situated between $\Gamma$ and the real axis, we get 

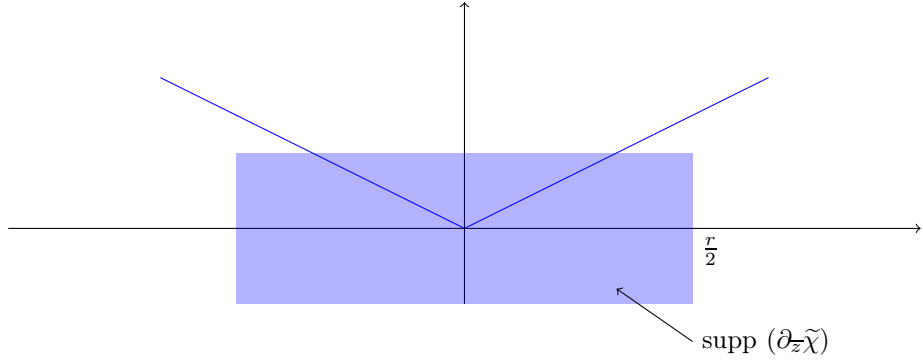
\begin{figure}
 \begin{center}
\begin{tikzpicture}[scale=2]
\draw[blue](-2, 1) -- (0,0)--(2,1); 
\draw[->] (-3,0) -- (3,0);
\draw[->] (0,-0.5) -- (0,1.5);
\draw (1.5, -.15) node[right] {$\frac {r } 2 $};
\fill[blue, opacity = .3] (-1.5, -.5) -- (-1.5, 0.5) -- (1.5,0.5) -- (1.5,-.5) -- cycle;
\draw[->] (1.5,-0.75) -- (1,-.4);
\draw (1.5, -.75)  node[right] {supp ($\partial_{\overline{z}}\widetilde{\chi}$)};
\end{tikzpicture}
\end{center}
\caption{A first contour $\Gamma$.}\label{fig.0}
\end{figure} 
\begin{equation}
\label{5.1ter}
e^{-itP}\frac{\chi (P)}{(1-iP)^k}x
= \frac{1}{2i\pi}\int _{ \Gamma}e^{-it \xi}\frac{1}{2i\pi}\frac{ \widetilde{\chi} (\xi)}{\bigl(1-i\xi\bigr)^ k} \frac 1 {\xi +P}d\xi\ x 
-  \frac{1}{\pi}\int _{ \Omega\cap\Gamma_-}e ^{-it\xi}
\frac{\partial_{\overline{z}}\widetilde{\chi}  ( \eta)}{\bigl(1-i\xi\bigr)^ k} \frac 1 {\xi +P}|d\eta|\ x,
\end{equation}
which is estimated in norm by 
$$ C \int_0^ r e^{-t\delta \eta} d\eta+ C_N\int_{ \Omega\cap\Gamma_-} e^{- t \Im (\xi) } \Im (\xi)^N d \xi \leq \frac C t.
$$
It remains to estimate the contribution of $(1- \chi (P))$.

For $ U_0\in {\mathcal H}$ put $ V=
e^{-itP}\chi(x)U_0$. 

We observe that for any  $s\geq 0$
\begin{equation}
\label{5.0}
 \ \| V\bigl(s\bigr)\| \leq \| U_0 \|.
\end{equation}
Using Fourier-Laplace transform, for any  $t>0$, $x\in {\mathcal H}$ we have
\begin{equation}
\label{5.1}
e^{-itP}\frac{1- \chi (P)}{(1-iP)^k}x= \frac{1}{2i\pi}\int _{ \xi =-\infty-i0} ^{+\infty -i0}e ^{it\xi}
\frac{1- \chi ( \xi)}{\bigl(1-i\xi\bigr)^ k} \frac 1 {\xi +P}d\xi\ x.
\end{equation}
The idea of the proof consists in decomposing in~\eqref{5.1} the integral into two parts,  the first one corresponding to "low" frequencies, $|\xi|\leq \log^2 (t)$,
 for which we perform contour deformations (after taking a suitable quasi-analytic extension of $\chi$), using the assumptions on the resolvent, while the other part corresponds to "high" frequencies for which we simply use that $|(1-i\xi )^ {-k}|\leq \log (t)^{-2k} $.

 In order to be able to perform this decomposition and still use contour deformation arguments, we take a convolution
 with a gaussian function and perform the decomposition on the convolution side.

 Furthermore, instead of decomposing
 the Cauchy data (which would introduce a $\delta_{t=0}\notin L^{1}_{\text{loc}}(\mathbb{R}_{t})$ singularity at $0$), it is better to transform the equation into a non-homogeneous one and perform the decomposition on the right hand side. 

Let~$ \psi \in C^{\infty}\bigl(\mathbb{R}_t\bigr)$, equal to 
$ 0 $ for~$ t<1/3$ and to~$ 1$ for~$ t>2/3$,  and let $ U= \frac 1 {(1-iP) ^
  k}\psi V$ be solution of 
\begin{align} \bigl(\partial _t +iP\bigr)U &= \psi '\bigl(t\bigr)
\frac{1- \chi (P)}{(1-iP)^k}V\bigl(t\bigr).\\
\intertext{We have}
U(t)&=\int_{0}^{t} e^{-i(t-s)P}\psi'(s)\frac{1- \chi (P)}{(1-iP)^k}V\bigl(s\bigr) ds.
\label{eq5.789}
\end{align}
\par\par
Let~$ c_0>0$ and~$ c_1>0$ to be chosen later.

 We have
\begin{equation}
\begin{split}
U(t) &=\int _{s=0} ^t ds\int_{\Ima \xi = 0^-}d\xi\psi '(s)
e^{i\bigl(t-s\bigr)\xi}\frac{1- \chi ( \xi)}{(1-i\xi)^k}\frac 1 {\xi +P}V\bigl(s\bigr)\int_{-\infty}
^{+\infty}\sqrt{\frac{c_0}{\pi}}e^{-c_0 \bigl(\lambda-
\frac{\xi}{{\log t}}\bigr)^{2}} d\lambda\\
&= \int_s\int_\xi \int_{|\lambda|<c_1{\log t}}+\int_s\int_\xi
\int_{|\lambda|\geq c_1 {\log t}} \\
&= I_1 + I_2.
\end{split} 
\end{equation}
\subsubsection{Estimating~$I_1$}
\label{se.3.3.1}
To estimate the contribution of the term  $I_{1}$ the idea is to deform the integration contour
 in the ~$\xi$ variable into the contour $\Gamma $ of Figure~\ref{fig.1} defined by 
 $$ \Gamma= \left[-\frac 1 2, \frac 1 2 \right] \cup \left[ - \frac 1 2 , -\frac 1 2 + i \varepsilon e^{-a\sqrt{1/2}}\right]
 \cup \left[ - \frac 1 2 , -\frac 1 2 + i \varepsilon e^{-a\sqrt{1/2}}\right] \cup \widetilde{\Gamma},
 $$ 
 $$ \widetilde{\Gamma} = \left\{ x+\varepsilon e^{-a\sqrt{|x|}}, |x|\geq \frac 1 2\right\}.$$
 
 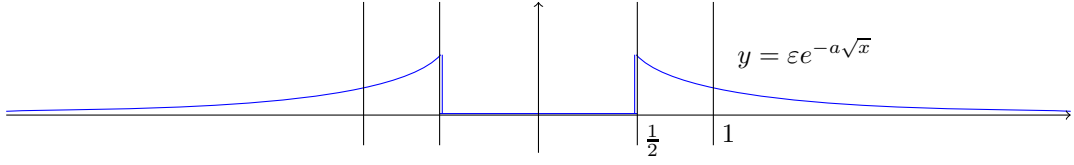
\begin{figure}
 \begin{center}
\begin{tikzpicture}[scale=1]
\draw[blue](-1.298, 0.02) -- (1.298, 0.02); 
\draw[blue](1.265, 0.02) -- (1.265, .8); 
\draw[blue] (1.285,.8).. controls (2,0)and  (6,0.1) .. (7, 0.05);
\draw[blue](-1.265, 0.02) -- (-1.265, .8); 
\draw[blue] (-1.285,.8).. controls (-2,0)and  (-6,0.1) .. (-7, 0.05);
\draw[->] (-7,0) -- (7,0);
\draw[->] (0,-0.5) -- (0,1.5);
\draw (3.5,.5) node[above] {$y=\varepsilon e^{-a\sqrt{x}}$};
\draw (1.3, -.4) -- (1.3, 1.5);
\draw (-1.3, -.4) -- (-1.3, 1.5);
\draw (2.3, -.4) -- (2.3, 1.5);
\draw (-2.3, -.4) -- (-2.3, 1.5);
\draw (1.5, 0) node[below] {$\frac 1 2$};
\draw (2.5, 0) node[below] {$1$};
\end{tikzpicture}
\end{center}
\caption{A second contour $\Gamma$.}\label{fig.1}
\end{figure} To be able to perform such a 
deformation, we first replace the function $\chi$ by a quasi-analytic extension, i.e. a function 
$$ \widetilde{\chi} : z \in \mathbb{C} \mapsto \widetilde{\chi} ( z),$$
which coincides with $\chi$ on $\mathbb{R}$, is supported in $\left\{ \Rea z \in (-1,1)\right\}$, equal to $1$ in $\left\{ \Rea z \in (-\frac 1 2 ,\frac 1 2)\right\}$ and such that 
$$ \forall N \in \mathbb{N}, \exists C_N; \forall z \in \mathbb{C}, |\partial_{\overline{z}} \chi (z) |\leq C_N |\Im z| ^N .$$
Then we have to check that for any $s\in [0,1]$ the operator $\chi(x) R(\xi)
e^{-isP}\chi(x)$ is holomorphic with respect to the $\xi$ variable in the domain below $\Gamma$
 and satisfies (uniformly with respect to $s$) an estimate of the type~\eqref{e2}.

 For $\Ima \xi<0$ the two families of operators 
\begin{equation}
\label{eq853}
e^{is\xi} \bigl( \chi(x)R(\xi)\chi(x)-i \int _{0}^{s} \chi(x)e^{
    -i\sigma(P+\xi)}\chi(x) d\sigma \bigr)\quad \text{and} \quad \chi(x) R(\xi )e^{-isP}\chi(x),
\end{equation}
coincide for  $s=0$, and satisfy the same differential equation
\begin{equation}
\partial _{s} w = i\xi w - i\chi(x) e^{-isP} \chi(x),
\end{equation}
hence, these two families are equal for $\Ima
\xi <0$, and the family on the l.h.s. gives the proper analytic continuation.  As a consequence, it is possible
 in the integral defining  $I_{1}$ to deform the contour into the contour $\Gamma$ (notice that the deformation
 near infinity is justified, for fixed $\lambda$, by the fast decay of the integrand ).

 According to Green's formula (notice that according to the support condition on $\chi$, the only contour contribution is given by $\widetilde{\Gamma}$), we get that 
 $$ I_1 =  \int_s\int_{ \xi\in \widetilde{\Gamma}} \int_{|\lambda|<c_1{\log t}}+i \int _{s=0} ^t ds\int_{z\in \Omega} \psi '(s)
e^{i\bigl(t-s\bigr)\xi}\frac{\partial_{\overline{z}}  \chi ( z)}{(1-iz)^k}\frac 1 {z +P}V\bigl(s\bigr)\int_{-\infty}
^{+\infty}\sqrt{\frac{c_0}{\pi}}e^{-c_0 \bigl(\lambda-
\frac{z}{{\log t}}\bigr)^{2}} |z| d\lambda,
$$ 
where 
$$ \Omega= \text{supp} (  \partial_{\overline{z}}  \chi ) \subset \left\{ x+ iy; |x| \in \left[\frac 1 2 , 1\right] , 0\leq y \leq  \varepsilon e^{-a\sqrt{|x|}}\right\}.$$

From~\eqref{5.0}, using the fact that the operator~$\chi(x) R(\xi)
e^{-isP}\chi(x)$ is uniformly bounded on  $H$ with respect to ~$ z\in \Omega $ and $s\in [0,1]$ for any~$ t\geq 2$, we get that the contribution of the second term is, for any $K>0$,  bounded by 
\begin{equation}
\label{5.2}
\Big\|\chi(x)\int_s \int_{z \in \Omega}\int_{|\lambda|<c_1
 \log(t)} \Big\| \leq C_N \|U_0\| \int_{z \in \Omega} e
^{-\bigl(t-1\bigr)\Rea z} (\Rea z)^K dz \leq \frac { C_K\|U_0\|}{  (t-1)^{K+1}}.
\end{equation}
From~\eqref{e2}, we also have for any~$ t>1$
\begin{multline}
\label{5.3}
\Bigl\|\chi(x)\int_s\int_{\xi \in \widetilde{\Gamma}} \int_{|\lambda|<c_1
 \log(t)}\Bigr\|\\
\leq  C \sqrt{c_0}\int_{\eta = -\infty}^{+\infty}\int_{|\lambda|<c_1
 \log(t)}e ^{-(t-1)\varepsilon _0 e ^{-a
    |\eta|^{1/2}}+A|\eta|^{1/2}-c_0(\lambda-\eta/{\log t})^2}d\eta\ d\lambda\|U_0\|.
\end{multline}
Choose now~$ c_2>0$ such that~$\sqrt{c_2} a<1$, and
$$ \varphi= -(t-1)\varepsilon _0 e ^{-a
    |\eta|^{1/2}}+A|\eta|^{1/2}-c_0(\lambda-\eta/{\log t})^2.$$ Then if $|\eta|\leq c_{2} \log^2 t$, we also have 
\begin{equation}
\label{5.5}
\varphi \leq \sqrt{c_2} A \log t
-(t-1)\varepsilon_0 t^{-\sqrt{c_2}a}.
\end{equation}
Choose~$ c_1 \in ]0, c_2[$. There exists~$ \delta >0$ such that for any~$|\lambda|<c_1
 \log(t)$, if~$|\eta|> c_2 \log^2 t$ then one has
\begin{align}
(\lambda - \eta
  /{\log t})^2 &\geq \delta (\lambda ^2 + (\eta
  /{\log t})^2),\\
\intertext{then}
\varphi &\leq A|\eta|-
  c_0 \delta (\lambda ^2 + (\eta
  /{\log t})^2).
\end{align}
 Choose finally~$c_0 >\frac{A}{\delta c_2}+1$. There exists ~$\varepsilon>0$ such that
\begin{equation}
\label{5.4}
\int_{|\eta|> c_2 \log t}e ^{A|\eta|^{1/2}-
  c_0 \delta  (\eta
  /\log(t))^2}= {\mathcal {O}}\bigl(e ^ {-\varepsilon \log(t)}\bigr).
\end{equation}
 After~(\ref{5.2}),~(\ref{5.3}),~(\ref{5.5}) and~(\ref{5.4}) we get
\begin{equation}
\label{5.6}
\left\|\chi _1I_1\right\|\leq C t^{-\varepsilon} \|U_0\|.
\end{equation}
\subsubsection{Estimating~$I_2$}
Let
\begin{equation}
\label{5.21}
J(u) = \int_{s=0}^1\iint_{\genfrac{}{}{0pt}{}{\Ima \xi = 0^-}{|\lambda|\geq c_1\log(t)}} 
\psi '\bigl(s\bigr)e^{i\bigl(u-s\bigr)\xi}\frac{1}{(1-i\xi)^k}\frac 1 {\xi+P}V\bigl(s\bigr)\sqrt{\frac{c_0}{\pi}}e^{-c_0
  \bigl(\lambda- \frac{\xi}{\log(t)}\bigr)^{2}}.
\end{equation}
For~$t\geq 1$, we have~$J(t) = I_2 (t)$ and for any~$ u\in \mathbb{R}$,
\begin{equation}
\label{5.20}
(\partial _u +iP)J(u) = \int_{s=0}^1\iint_{\genfrac{}{}{0pt}{}{\Ima \xi = 0^-}{|\lambda|\geq c_1\log(t)}}  
\psi '\bigl(s\bigr)\frac {ie^{i\bigl(u-s\bigr)\xi}}{\bigl(1-i\xi\bigr)^ k}V\bigl(s\bigr)\sqrt{\frac{c_0}{\pi}}e^{-c_0
  \bigl(\lambda- \frac{\xi}{\log(t)}\bigr)^{2}}:= K(u),
\end{equation}
which implies
 \begin{equation}
\label{5.17}
J(t) = e^{-itP} J(0) + \int_{0}^t e^{-i(t-s)P}K(s) ds.
\end{equation}
\par
To estimate the norm of~$ J\bigl(t\bigr)$ in~$ H$, we shall use the fact that 
$ e ^{-isP}$ is a contraction on $H$ for  $s\geq 0$, and we bound separately~$ K\bigl(u\bigr)$ for~$ u\geq 1$,~$ J\bigl(0\bigr)$ and~$
\int_{0}^ 1 \|K\bigl(u\bigr)\| du$.

 Since the r.h.s. in~\eqref{eq5.789} (the analog of $K(u)$) vanishes for  $u>1$ as well as the Cauchy data ($U(0)$), we expect the main
 contribution to come from the third term.\par
For~$u \in [1,t]$, we deform in~\eqref{5.20}  the contour in the ~$\xi$ variable into the contour defined by~$\Ima \xi =\log(t)$,
 which gives according to~\eqref{5.0}  using~$ k>1 $ and 
$\text{supp}(\psi') \subset
[1/3, 2/3]$
\begin{equation}
\begin{split}
\label{5.7}
\|K(u)\|& \leq \int_{-\infty}^{+\infty} e^{-(u-2/3)\log(t)}\frac
1{(1+|\xi|)^k}d \xi \|U_0\|\\
& \leq C_ke^{-\log(t)/3}\|U_0\|.
\end{split}
\end{equation}\par
Now we bound~$J(0)$.

 Let us study for example the contribution to~(\ref{5.21}) 
of the region~$\lambda >c_1\log(t)$.

 We deform the contour into the contour~$\Gamma = \Gamma^+ \cup
\Gamma^-$, with
\begin{equation}
\begin{split}
\Gamma^+ &= \{z = 1+\eta -i\log(t); \ \eta>0\},\\
\Gamma^- &= \{z = 1+\eta -i/2; \ \eta\leq 0\}\cup [1-i/2,
1-i\log(t)].
\end{split}
\end{equation}
For~$\xi\in \Gamma^-$, we have for any~$ s\in
\left[0,1\right]$ and any~$ \lambda \in \left[c_1 \sqrt {\log t}, + \infty\right]$
\begin{equation}
\relax
\Big\|e^{-is \xi}\frac{V\bigl(s\bigr)}{\bigl(1-i
\xi\bigr)^ k}\frac 1 {\xi +P}\sqrt{\frac{c_0}{2\pi}}e^{-c_0
  \bigl(\lambda- \frac{\xi}{\log(t)}\bigr)^{2}}\Big\|\leq
\frac C{(1+|\xi|)^k}e^{-\delta(\lambda^2 +\xi^2/\log^2 t)}\|U_0\| \ \text{($ \delta>0$)}.
\end{equation}
As a consequence we deduce that the contribution of~$\Gamma^-$ to~$ J(0)$ is bounded in norm by
\begin{equation}
\label{5.14}
C\log(t)\int _{\lambda\geq c_1\log(t)} e^{-\delta
  \lambda ^2}\|U_0\|= {\mathcal {O}}\bigl(e^{-\varepsilon \log^2 t}\bigr)\|U_0\|.
\end{equation}
For~$\xi \in \Gamma^ +$ and $s\in [1/3, 2/3]$ we have
\begin{equation}
\relax
\Big\|e^{-is\xi}\frac{V\bigl(s\bigr)}{\bigl(1-i
\xi\bigr)^ k}\frac 1 {\xi +P}\sqrt{\frac{c_0}{\pi}}e^{-c_0
  \bigl(\lambda- \frac{\xi}{\sqrt{\log
        t}}\bigr)^{2}}\Big\|\leq C e^{-\log(t)/3}\frac
1{(1+|\eta|)^k}e^{-
c_0\bigl(\lambda -\frac \eta {\log(t)}\bigr)^2}\|U_0\|,
\end{equation}
hence, using~$k>1$, the contribution of ~$\Gamma^+$ to~$ J(0)$ is bounded in norm by
\begin{equation}
\label{5.15}
Ce^{-\log(t)/3}\|U_0\|.
\end{equation}
The contributions to ~$ J\bigl(0\bigr)$ of the region~$ \lambda < -c_1\log(t)$ are bounded similarly.
\par
Finally, it remains to bound
\begin{equation}
\label{5.12}
\int_{0}^1 \|K(u)\| du\leq \left(\int_{0}^1 \|K(u)\|^2 du\right)^{1/2}.
\end{equation}
From  Plancherel formula, 
\begin{equation}
\label{5.9}
\begin{split}
\int_{-\infty}^{+\infty} \|K(u)\|^2 du&= C
\int_{-\infty}^{+\infty}\Big\| {\frac i {\bigl(1-i
\xi\bigr)^k} \widehat{V\psi '}(\xi)\int_{|\lambda|\geq c_1\log(t)}e^{-c_0
  \bigl(\lambda- \frac{\xi}{{\log
        t}}\bigr)^{2}}d\lambda }\Big\|^2 d\xi\\
&= C\int_{-\infty}^{+\infty}\|H(\xi)\|^2 d \xi,
\end{split}
\end{equation}
where, for any~$|\xi| >\frac {c_1}{2} \log^2 t$
\begin{equation}
\label{5.10}
\begin{split}
\|H(\xi)\|& =\Big\| \int_{|\lambda|\geq c_1\log(t)}\frac
1{\bigl(1-i \xi\bigr)^ k}e^{-c_0
  \bigl(\lambda- \frac{\xi}{{\log
        t}}\bigr)^{2}} d\lambda \widehat{V\psi '}\bigl(\xi\bigr)\Big\|\\
& \leq \frac C {(\log t)^{2k}}
  \|\widehat{V\psi '}(\xi)\|,
\end{split}
\end{equation}
and for~$|\xi| \leq \frac{c_1}{2}\log^2t$
 \begin{equation}
\label{5.11}
\|H(\xi)\| \leq \int _{|\lambda| > c_1\log(t)}e^{-\delta (\lambda^2 + \xi ^2/\log^2 t)}
 d\lambda \|\widehat{V\psi '}(\xi)\|\leq C e^{-\epsilon \log^2 t}\|\widehat{V\psi '}(\xi)\|. 
\end{equation}
Then from~(\ref{5.12}),~(\ref{5.9})~(\ref{5.10}) and~(\ref{5.11}), using
\begin{equation}
\int_{-\infty}^{+\infty} \|\widehat {V\psi '}\bigl(\xi\bigr)\|^ 2 d\xi =
\int_{-\infty}^{+\infty} \|\psi ' V\bigl(s\bigr)\|^ 2 ds\leq C \int_0^1 |\psi'
\bigl(s\bigr)|^2 ds \|U_0\|^ 2,
\end{equation}
(let us recall that $V(s) = e^{is B}\chi(x) U_{0}\Rightarrow \| V(s)\|\leq
\|\chi(x) U_{0}\|\leq C\|U_{0}\|$), we deduce that
\begin{equation}
\label{5.18}
\int_{0}^1 \|K(u)\| du\leq C \bigl(\frac 1 {\bigl(\log t\bigr)^ {2k}} + e
^ {- \varepsilon \log t}\bigr)\|U_0 \|.
\end{equation}
From~(\ref{5.17}),~(\ref{5.7})~(\ref{5.14}),~(\ref{5.15}) and~(\ref{5.18}), we obtain finally 
\begin{equation*}
\left\|I_2\right\|\leq \frac C{(\log t)^{2}}\|U_0 \|,
\end{equation*}
which ends the proof of Theorem~\ref{t3}.

\section{The case of a weakly trapping obstacle. Proof of Theorem~\ref{t4}}

Theorem~\ref{t4} exhibits two quite different regimes: 
\begin{enumerate}
\item the first one ($-\frac 1 2 \leq N <0$) corresponds to non trapping ($N=- 1/2$) or weakly trapping ($-\frac 1 2 <N<0$)
 situations for which the (space truncated) Schr\"odinger evolution is smoothing and consequently we gain derivatives,
\item the second one ($0<N$) corresponds to stronger trapping, for which this smoothing effect is no more true.
\end{enumerate} 
We shall first focus on this second regime $N>0$ and give the proof of~\eqref{esti} under the  assumption (E).

\subsection{Analytic properties of the resolvent}
%In this section we assume that the metric is euclidean near infinity (E). 

Adapting a result of Tsutsumi (see Lemma 2.3 in \cite{T1}) to the euclidean case (E) with $\varrho(x)=1$ and $a_{ij}(x)=\delta_{ij}$ for $|x|>R$, we have
%Recall the following result of Tsutsumi \cite{T1} (following ideas of Vainberg \cite{V} \cite{V2}).
\begin{prop}
\label{pv}
%Let $R>a$ and $d\geq 3$.
Let $R>a$, $d\geq 3$ and $H^2_e(\Omega)$ the weighted space $H^2_e(\Omega):=\{u: e^{-|x|^2}u\in H^2(\Omega)\}$
 There exist a positive number $K$ such that
\begin{enumerate}
\item If $d$ is odd, 
\begin{equation}
{\mathcal R}(z)=\sum_{j=0}^{\infty} B_{2j} z^{2j}+\sum_{j=\frac{d-3}{2}}^{\infty} B_{2j+1} z^{2j+1},
\label{e7}
\end{equation}
in the region ${\mathcal W}=\{ z\in {\mathcal O}\ ;\ |z|\leq K\}$, where the operators $B_j$, $j=0,1,2,...$ are bounded from $L^2_R(\Omega)$
to $H^2_e(\Omega)$ and the expansions converge uniformly in operator norm.
\item If $d$ is even, 
\begin{equation}
{\mathcal R}(z)=\sum_{m=0}^{\infty} \sum_{j=0}^{\infty} B_{m,j}\bigl(z^{d-2}\log z\bigr)^m z^{2j},
\label{e8}
\end{equation}
in the region ${\mathcal W}=\{ z\in {\mathcal O}\ ;\ |z|\leq K,\ -\frac{\pi}{2}<\mbox{Arg z}<\frac{3\pi}{2}\}$,  where the operators $B_{mj}$, $m,j=0,1,2,...$
are bounded from  $L^2_R(\Omega)$
in $H^2_e(\Omega)$ and the expansions also converge uniformly in operator norm.
\end{enumerate}

\end{prop}
\vskip0.25cm 
{\bf Proof:}
After \cite{V} we define the operator $G(z)\in \mbox{Hom}(L^2_R(\Omega),H^2_e(\Omega))$ by
\begin{equation}
G(z)w:=\beta_1 R_a(z_0)(\alpha_1 w)+\beta_2 R_0(z)(\alpha_2 w),
\label{pv1}
\end{equation}
where $\alpha_{1,2}$ and $\beta_{1,2}$ are two $C^{\infty}({\mathbb R}^d)$ cut off such that
\[
\alpha_1(x)=
\left|  \begin{array}{ll}
{\displaystyle 0\ \ \ \mbox{if}\ |x|<R+1/2},\\\\
{\displaystyle 1\ \ \ \mbox{if}\ |x|>R+2/3},
 \end{array}
\right.
\]
$\alpha_2(x)=1-\alpha_1(x)$,
\[
\beta_1(x)=
\left|  \begin{array}{ll}
{\displaystyle 1\ \ \ \mbox{if}\ |x|<R+2/3},\\\
{\displaystyle 1\ \ \ \mbox{if}\ |x|>R+1},
 \end{array}
\right.
\]
and
\[
\beta_2(x)=
\left|  \begin{array}{ll}
{\displaystyle 0\ \ \ \mbox{if}\ |x|<R},\\\
{\displaystyle 1\ \ \ \mbox{if}\ |x|>R+1/3}.
 \end{array}
\right.
\]
Consider
\begin{equation}
S(z):=(z-P)G(z)-I,
\label{pv2}
\end{equation}
where $I$ is the identity operator.

 After \cite{V} we know that $S(z)$ is compact from $L^2_R(\Omega)$ to $L^2_R(\Omega)$ for any $z$. Moreover
$(z-P)^{-1}S(z)=G(z)(I+S(z))^{-1}$, where $z\to (I+S(z))^{-1}$ is meromorphic.

 Now as we know \cite{V} that expansions (\ref{e7}) and (\ref{e8}) are valid
 for the free operator $R_0(z)(\alpha_2\ \cdot)\equiv (\Delta+z^2)^{-1}$, it follows that they also hold for $G(z)$ and $S(z)$.

 As $I+S(0)=\Delta G(0)$ in $\Omega_R$ we see that $G(0)w=R_0(0)w\to 0$ for $|x|\to \infty$. Moreover as the solution
of
\[
\left\{ 
 \begin{array}{lll}
{\displaystyle P_0 u=0\ \ \ \mbox{in}\ \Omega},\\\
{\displaystyle \left. u\right|_{\partial\Omega}=0},\\\
{\displaystyle u\to 0\ \ \ \mbox{as}\ |x|\to\infty}.
 \end{array}
\right.
\]
is $u\equiv 0$, we get from Fredholm theory that a necessary and sufficient condition for $I+S(0)$ to have a bounded inverse is that $G(0)$ is one to one.

Applying now formally the Neumann expansion
\[
(I+S(k))^{-1}=\sum_{j\geq 0}(-1)^k \left[(I+S(0))^{-1}(S(z)-S(0))\right]^k(I+S(0))^{-1},
\]
near $z=0$, we conclude that ${\mathcal R}(z)$ satisfies (\ref{e7}) and (\ref{e8}), if and only if $G(0)$ is one to one.

In order to show that suppose that $w\in L^2_R(\Omega)$ satisfying $G(0)w=0$. From the previous definitions of $\alpha_{1,2}$ and $\beta_{1,2}$ we have
\[
\left\{ 
 \begin{array}{lll}
{\displaystyle R_a(z_0)(\alpha_1w)=0\ \ \ \mbox{for}\ |x|<R},\\\
{\displaystyle R_0(0)(\alpha_2w)=0\ \ \ \mbox{for}\ |x|>R+1},
 \end{array}
\right.
\]
so
\[
\left\{ 
 \begin{array}{lll}
{\displaystyle w(x)=\alpha_1(x) w(x)=(P-z_0)R_a(z_0)(\alpha_1w)=0\ \ \ \mbox{for}\ |x|<R},\\\
{\displaystyle w(x)=\alpha_2(x) w(x)=P_0 R_0(0)(\alpha_2w)=0\ \ \ \mbox{for}\ |x|>R+1}.
 \end{array}
\right.
\]
Using the fact that
\begin{equation}
R_a(z_0)(\alpha_1 w)+ R_0(0)(\alpha_2 w)=0\ \ \ \mbox{for}\ R+1/3<|x|<R+2/3,
\label{pv3}
\end{equation}
we get
\begin{equation}
(P_0+\zeta_1(x))R_a(z_0)(\alpha_1 w)=0\ \ \ \mbox{for}\ R+1/3<|x|<a,
\label{pv4}
\end{equation}
\begin{equation}
R_a(z_0)(\alpha_1 w)=0\ \ \ \mbox{for}\ |x|=a,
\label{pv5}
\end{equation}
and
\begin{equation}
(P_0+\zeta_2(x))R_0(0)(\alpha_2 w)=0\ \ \ \mbox{for}\ |x|<R+2/3,
\label{pv6}
\end{equation}
where
\begin{equation}
\zeta_1(x)=
\left|  \begin{array}{ll}
{\displaystyle -z_0\ \ \ \mbox{if}\ R+1/2<|x|\leq a},\\\
{\displaystyle 0\ \ \ \mbox{if}\ R+1/3<|x|<R+1/2},
 \end{array}
\right.
\label{pv7}
\end{equation}
and
\begin{equation}
\zeta_2(x)=
\left|  \begin{array}{ll}
{\displaystyle 0\ \ \ \mbox{if}\ |x|<R+1/2},\\\
{\displaystyle -z_0\ \ \ \mbox{if}\ R+1/2<|x|<R+2/3}.
 \end{array}
\right.
\label{pv8}
\end{equation}
From (\ref{pv4}) and (\ref{pv6}) we see that $R_a(z_0)(\alpha_1 w)$ is H\"older continuous for $ R+1/3<|x|\leq a$, and that
$R_0(0)(\alpha_2 w)$ is H\"older continuous for $|x|\leq R+2/3$.
Therefore after (\ref{pv3})  $R_a(z_0)(\alpha_1 w)$ and $R_0(0)(\alpha_2 w)$ are also H\"older continuous for $|x|\leq a$.

Applying the maximum principle \cite{GT}, we get
\[
\max \{|R_a(z_0)(\alpha_1 w)|;|x|=R+1/3\}\geq \max \{|R_a(z_0)(\alpha_1 w)|;|x|=R+2/3\},
\]
and
\[
\max \{|R_0(0)(\alpha_2 w)|;|x|=R+2/3\}\geq \max \{|R_0(0)(\alpha_2 w)|;|x|=R+1/3\}.
\]
But from (\ref{pv3}) it follows that in fact
\[
\max \{|R_a(z_0)(\alpha_1 w)|;|x|=R+1/3\}= \max \{|R_a(z_0)(\alpha_1 w)|;|x|=R+2/3\},
\]
then applying the maximum principle to $R_0(0)(\alpha_2 w)$ in $\{ R+1/3<|x|\leq a\}$ and using (\ref{pv5}) we obtain
\[
R_a(z_0)(\alpha_1 w)=0\ \ \ \mbox{if}\ R+1/3<|x|<a.
\]
Finally using equation $G(0)w=0$ together with (\ref{pv8}), we get
\[
R_0(0)(\alpha_2 w)=0\ \ \ \mbox{in}\ {\mathbb R}^d,
\]
and finally that
\[
R_a(z_0)(\alpha_1 w)=0\ \ \ \mbox{in}\ \Omega_R,
\]
so we conclude that $w$ vanishes identically and so $G(0)$ is actually one to one, which ends the proof.$\ \ \ \ \ \ \ \ \ \ \ \ \ \Box$

\subsection{Proof of Theorem~\ref{t4} (case $N>0$)}

In order to prove the decay, we shall use, as previously, a temporal cut-off with a contour deformation in the complex plane. 

For any $ U_0\in {\mathcal H}$, let 
$$ V=
 \frac 1 {(1-iP)^k} e^{-itP}\chi(x)U_0.
$$
As previously \eqref{5.0} and~\eqref{5.1} are still valid.

 We start from~\eqref{eq5.789}, and write \begin{equation}
\begin{split}
\chi(x) U(t) &=\chi(x)\int _{s=0} ^t ds\int_{\Ima \xi = 0^-}d\xi\psi '(s)
e^{i\bigl(t-s\bigr)\xi}\frac{1}{(1-i\xi)^k}\frac 1 {\xi +P}V\bigl(s\bigr)\\
& \quad \qquad \qquad \qquad \hfill \int_{-\infty}
^{+\infty}\sqrt{\frac{c_0}{\pi}}e^{-c_0 \bigl(\lambda-
\xi \bigl(\frac{\log (t)}{t} \bigr) ^{\frac 1 {2N}}\bigr)^{2}} d\lambda\\
&= \chi(x) \int_s\int_\xi \int_{|\lambda|<c_1\bigl(\frac{t}{\log(t)}\bigr) ^{\frac 1 {2N}}}+\chi(x) \int_s\int_\xi
\int_{|\lambda|\geq c_1 \bigl(\frac{t}{\log(t)}\bigr) ^{\frac 1 {2N}}}\\
&= I_1(U_0) + I_2(U_0).
\end{split} 
\end{equation}

\subsubsection{Estimate of~$I_1$}
\label{se.4.1}

Following Section~\ref{se.3.3.1}, it is possible in the integral defining $I_1$, to deform the integration contour on the  new contour $\Gamma$ defined as
\begin{equation}
\Gamma=\left[0,D+\imath cD^{-N}\right]
\cup
\left[0,-D+\imath c D^{-N}\right]
\cup
\left\{ z=\rho+\imath c \rho^{-N}:|\rho|\geq D\right\}
= \Gamma_1 \cup \Gamma_2 \cup \Gamma_3.
\label{e5bis}
\end{equation}
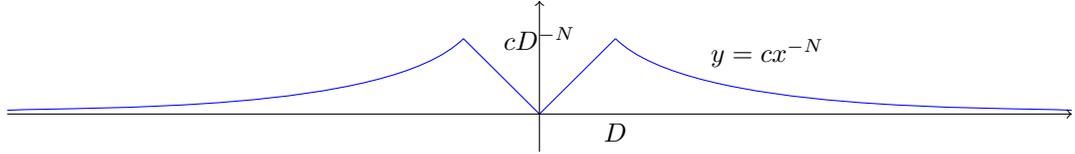
\begin{figure}[ht]

\begin{center}
\begin{tikzpicture}[scale=1]
\draw[blue](0, 0) -- (1, 1); 
\draw [blue](0, 0) -- (-1, 1) ;
\draw[blue] (1,1).. controls (2,0)and  (6,0.1) .. (7, 0.05);
\draw[blue] (-1,1).. controls (-2,0)and  (-6,0.1) .. (-7, 0.05);
\draw[->] (-7,0) -- (7,0);
\draw[->] (0,-0.5) -- (0,1.5);
\draw (1, 0) node[below] {$D$};
\draw (3,.5) node[above] {$y=cx^{-N}$};
\draw (0,1) node {$cD^{-N}$};
\end{tikzpicture}
\end{center}

\caption{ A third contour $\Gamma$.}\label{fig.3}
\end{figure}
Let us first estimate the contributions of $\Gamma _1 \cup \Gamma_2$.

 We have to be more careful here than in Section~\ref{se.3.3} because this analysis will give most of the time the leading term 
(and an $\mathcal{O}(t^{-\epsilon})$ estimate  is no more sufficient).

 We distinguish two cases according whether $d$ is odd or even.

\begin{itemize}

\item $d$ odd.

From~(\ref{e7}), the main contribution of the operator
$\chi(x) R(\xi)
e^{isP}\chi(x)$ for $ \xi \in \Gamma _1 \cup
\Gamma _2$ and $s\in [0,1]$ for any~$ t\geq 2$, 
is $B_{d-2}\xi^{\frac{d-2}{2}}$ with $B_{d-2}$ bounded.

If $\Gamma_1=\{re^{i\theta_1},\ r\in [0,a]\}$, we have
\begin{equation}
\label{5.2bis}
\Big\|\chi(x)\int_s \int_{\xi \in \Gamma_1} \Big\|
 \leq C \|U_0\|
 \int_0^a e^{-tr\sin\theta_1} r^{\frac{d-2}{2}}\ dr \leq \frac { C\|U_0\|}{ t^{\frac{d}{2}}},
\end{equation}
and the same estimate is true on $\Gamma_2$.

\item $d$ even.

From~(\ref{e8}), we can check that the main contribution from 

$$\chi(x) R(\xi)
e^{isP}\chi(x)\ \ \mbox{for} \xi \in \Gamma _1 \cup
\Gamma _2,\ s\in [0,1]\ \mbox{and any}\ t\geq 2,$$
is $B_{1,0}\xi^{\frac{d-2}{2}}\log \xi$ where  $B_{1,0}$ is a bounded operator.

By taking the parametrization $\Gamma_1=\{re^{i\theta_1},\ r\in [0,a]\}$ et $\Gamma_2=\{re^{i\theta_2},\ r\in [0,b]\}$, we get
\begin{equation}
\label{5.2ter}
\left\|\chi(x)\int_s \int_{\xi \in \Gamma_1\cup\Gamma_2} \right\|
=\left\|\chi(x)\int_s\left(\int_{\xi \in \Gamma_1}+ \int_{\xi \in \Gamma_2}\right)\right\|,
\end{equation}
and we write
\[
A:=\int_{\xi \in \Gamma_1}+ \int_{\xi \in \Gamma_2}
=\int_{\Gamma_1} e^{it\xi} \xi^{\frac{d-2}{2}}\log \xi\ d\xi+\int_{\Gamma_2} e^{it\xi} \xi^{\frac{d-2}{2}}\log \xi\ d\xi.
\]
Let $\xi=z^2$, so that $A$ becomes
\[
A:=2\int_{\widetilde\Gamma_1} e^{itz^2} z^{d-1}\log z\ dz+2\int_{\widetilde\Gamma_2}e^{itz^2} z^{d-1}\log z\ dz,
\]
where $\widetilde\Gamma_j$, $j=1,2$ are (after deformation)  segments in the half plane $\{\Ima z>0\}$.

Putting $\zeta=t^{1/2}z$, we get
\[
A:=\frac{2}{t^{d/2}}\int_{\widetilde\Gamma_{1,t}\cup\widetilde\Gamma_{2,t}} e^{i\zeta^2} \zeta^{d-1}\left(\log \zeta-\frac{1}{2}\ \log t\right)\ d\zeta
=:A_1+A_2.
\]
As the integrand in $A_2$ which corresponds to the term $\frac{1}{2} \log t$ is no more singular at $0$, it is possible to deform the contour 
$\widetilde\Gamma_{1,t}\cup\widetilde\Gamma_{2,t}$ into a small horizontal segment $\{s+i\beta,\ c_1<s<c_2\}$, with $\beta>0$
 which gives
\begin{equation}\label{5.5quar}
|A_2|=\left|\frac{\log t}{t^{d/2}}\int_{\widetilde\Gamma_{1,t}\cup\widetilde\Gamma_{2,t}} e^{i\zeta^2} \zeta^{d-1}\ d\zeta\right|
\leq C\ \frac{\log t}{t^{d/2}}e^{-2\beta t}.
\end{equation}
Similarly the  contribution of $\Gamma_1$ to $A_1$ is bounded by
\begin{equation}\label{5.5quint}
 C \|U_0\|
 \int_0^a e^{-tr\sin\theta_1} r^{\frac{d-2}{2}}\ |\log r|\ dr \leq \frac { C\|U_0\|}{ t^{\frac{d}{2}}},
\end{equation}
and the same is true on  $\Gamma_2$. Finally
\[
|A_1|
\leq\frac{C\|U_0\|}{t^{d/2}},
\]
which gives
\begin{equation}\label{5.5six}
\Big\|\chi(x)\int_s \int_{\xi \in \Gamma_1\cup\Gamma_2} \Big\|
  \leq \frac { C\|U_0\|}{ t^{\frac{d}{2}}}.
\end{equation}
It just remains to estimate the contribution of $\Gamma_3$. 

From~\eqref{e2bis}, we also have for any~$ t>1$
\begin{multline}
\label{5.3bis}
\left\|\chi(x)\int_s\int_{\xi \in \Gamma_3} \int_{|\lambda|<c_1\bigl(\frac{t}{\log(t)}\bigr) ^{\frac 1 {2N}}}\right\|\\
\leq  C \sqrt{c_0}\int_{\eta = -\infty}^{+\infty}\int_{|\lambda|<c_1\bigl(\frac{t}{\log(t)}\bigr) ^{\frac 1 {2N}}}e ^{-(t-1)c|\eta|^{-N}+N \log( |\eta|)-c_0\Bigl(\lambda-
\eta\bigl(\frac{\log(t)}{t}\bigr) ^{\frac 1 {2N}}\Bigr)^2}d\eta\ d\lambda\|U_0\|.
\end{multline}
Let now~$ c_2>0$ to be fixed later and let 
$$ \varphi= -(t-1)c|\eta|^{-N}+N \log( |\eta|)-c_0\left(\lambda-\eta\left(\frac{\log(t)}{t}\right) ^{\frac 1 {2N}}\right)^2. $$
Then if $|\eta|\leq c_{2} \left(\frac{ t}{ \log(t)}\right)^{\frac 1 N}$, we also have 
\begin{equation}
\label{5.5bis}
\varphi \leq - (t-1) c_2 ^{-\frac 1 N} \frac{\log (t)} t + N ( \log(t) - \log( c_2))\leq -(c_2 ^{-\frac 1 N}-(N+1))  \log(t).
\end{equation}
Let us choose~$ c_1 \in ]0, c_2[$.

 There exists~$ \delta >0$ such that for any~$|\lambda|<c_1\left( \frac{t}{\log(t)} \right)^{\frac 1 {2N}}$ and
 if~$|\eta|\leq c_{2} \left( \frac{t}{\log(t)} \right)^{\frac 1 {N}}$, then 
\begin{align}
\left(
\lambda - \eta
  \left(\frac{\log(t)}{t}\right) ^{\frac 1 {2N}}\right)^2 &\geq \delta \left(\lambda ^2 + \left(\eta \left(\frac{\log(t)}{t}\right) ^{\frac 1 {2N}}\right)^2\right),\\
\intertext{which implies}
\varphi &\leq N \log (\eta)-
  c_0 \delta \left(\lambda ^2 + \eta^2\left(\frac{ \log(t)} t \right)^{\frac 1 {N}}\right).
\end{align}
 We deduce that there exists ~$\varepsilon>0$ such that
\begin{equation}
\label{5.3ter}
\int_{|\eta|>c_{2} \bigl( \frac{t}{\log(t)} \bigr)^{\frac 1 {N}}}e ^{A|\eta|^{1/2}-
  c_0 \delta  \Bigl(\eta
  \bigl(\frac{\log(t)} {t} \bigr) ^{\frac 1{2N}}\Bigr)^2}= {\mathcal {O}}\Bigl(e ^ {-\epsilon \bigl(\frac{ t} { \log(t)} \bigr)^{\frac 1 {2N}}}\Bigr).
\end{equation}
Taking $c_2>0$ small enough, from~(\ref{5.2bis}),~(\ref{5.3bis}),~(\ref{5.5bis}),~(\ref{5.3ter}),~(\ref{5.5quar}),~(\ref{5.5quint}),~(\ref{5.5six}), and~(\ref{5.3ter}) we get 
\begin{equation}
\label{5.6bis}
\left\|\chi _1I_1\right\|\leq C t^{-d} \|U_0\|.
\end{equation}

\end{itemize}

\subsubsection{Estimate of~$I_2$ for $k>1$}\label{4.2.2}
Let
\begin{multline}
\label{5.21bis}
J(u) = \int_{s=0}^1\iint_{\genfrac{}{}{0pt}{}{\Ima \xi = 0^-}{|\lambda|>c_1\bigl(\frac{t}{\log(t)}\bigr) ^{1/ {2N}}}} 
\psi '\bigl(s\bigr)e^{i\bigl(u-s\bigr)\xi}\frac{1}{(1-i\xi)^k}\frac 1 {\xi+P}
V\bigl(s\bigr)\sqrt{\frac{c_0}{\pi}}e^{-c_0
  \Bigl(\lambda- \xi\bigl(\frac{\log(t)} t \bigr)^{\frac 1 {2N}}\Bigr)^{2}}.
\end{multline}
For~$t\geq 1$, we have~$J(t) = I_2 (t)$ and for any~$ u\in \mathbb{R}$
\begin{multline}
\label{5.20bis}
(\partial _u +iP)J(u) \\
= \int_{s=0}^1\iint_{\genfrac{}{}{0pt}{}{\Ima \xi = 0^-}{|\lambda|>c_1\bigl(\frac{t}{\log(t)}\bigr) ^{1/ {2N}}}}  
\psi '\bigl(s\bigr)\frac {ie^{i\bigl(u-s\bigr)\xi}}{\bigl(1-i\xi\bigr)^ k}V\bigl(s\bigr)\sqrt{\frac{c_0}{\pi}}e^{-c_0
 \Bigl(\lambda- \xi\bigl(\frac{\log(t)} t \bigr)^{\frac 1 {2N}}\Bigr)^{2}}:= K(u),
\end{multline}
which implies
 \begin{equation}
\label{5.17bis}
J(t) = e^{-itP} J(0) + \int_{0}^t e^{-i(t-s)P}K(s) ds.
\end{equation}
\par
To estimate the norm of~$ J\bigl(t\bigr)$ in~$ H$, we use the fact that 
$ e ^{-isP}$ is a contraction on $H$ for  $s\geq 0$, and we bound separately~$ K\bigl(u\bigr)$ for~$ u\geq 1$,~$ J\bigl(0\bigr)$ and~$
\int_{0}^ 1 \|K\bigl(u\bigr)\| du$.

 Since the r.h.s. in~\eqref{eq5.789} (the analog of $K(u)$) vanishes for  $u>1$ as well as the Cauchy data $U(0)$,
 we expect the main contribution to be coming from the third term.\par
For~$u \in [1,t]$, we have 
$$ K(u) = \int_{s=\frac 1 3 }^{\frac 2 3} 
\psi '\bigl(s\bigr)V\bigl(s\bigr) Q(u,s) ds,
$$
and
\begin{equation}\label{5.200}
 Q(u,s)= \iint_{\genfrac{}{}{0pt}{}{\Ima \xi = 0^-}{|\lambda|>c_1\bigl(\frac{t}{\log(t)}\bigr) ^{1/ {2N}}}} \frac {ie^{i\bigl(u-s\bigr)\xi}}{\bigl(1-i\xi\bigr)^ k}\sqrt{\frac{c_0}{\pi}}e^{-c_0
  \bigl(\lambda- \xi \bigl(\frac{\log(t)} t \bigr)^{\frac 1 {2N}}\bigr)^{2}}.
\end{equation}
    We  deform in~\eqref{5.200}  the contour in the ~$\xi$ variable into the contour 
  $$\Gamma= \left\{ \xi + i \epsilon \left(\frac{ t} {\log(t) } \right) ^{\frac 1 N}  ; \xi \in \mathbb{R}\right\}.  $$
  Using that $k>1$, we  bound the integral by 
  $$  e^{-(u-s)\epsilon \bigl(\frac{ t} {\log(t) } \bigr) ^{\frac 1 N}+ c_0\epsilon^2 \bigl(\frac{ t} {\log(t) } \bigr) ^{\frac 1 N} },$$
  which for $\epsilon>0$ small enough decays faster than any polynomial in $t$ (because for $u\geq 1, s \in \left[\frac 1 3, \frac 2 3\right]$, $u-s > \frac 1 3$).
 
So we get 
\begin{equation}
\label{5.7bis}
\|K(u)\| \leq C_m t^{-m}.
\end{equation}
\par
Now we bound~
\begin{multline}
\label{5.21ter}
{\displaystyle J(0) = \int_{s=0}^1\iint_{\genfrac{}{}{0pt}{}{\Ima \xi = 0^-}{|\lambda|>c_1\bigl(\frac{t}{\log(t)}\bigr) ^{1/ {2N}}}} 
\psi '\bigl(s\bigr)e^{i\bigl(u-s\bigr)\xi}\frac{1}{(1-i\xi)^k}\frac 1 {\xi+P}V\bigl(s\bigr)}\\
{\displaystyle \sqrt{\frac{c_0}{\pi}}e^{-c_0 \Bigl(\lambda- \xi\bigl(\frac{\log(t)} t \bigr)^{\frac 1 {2N}}\Bigr)^{2}} d\xi\ d\lambda\ ds.}
\end{multline}

Let us study for example  the contribution to~\eqref{5.21ter} of the region $ c_1\bigl(\frac{t}{\log(t)}\bigr) ^{\frac 1 {2N}}< \lambda$
 (the other contribution being similar).

 We shall now  deform the contour in the $\xi$ variable in~\eqref{5.21ter} into the contour
\begin{multline*}
\Gamma =  \left\{z = 1+\eta -i/2; \ \eta\leq 0\right\}
\cup 
\left[1-i/2,1-i\epsilon \left(\frac {t} {\log(t)}\right)^{\frac 1 N}\right]
\cup 
\left\{z = 1+\eta -i\epsilon \left(\frac {t} {\log(t)}\right)^{\frac 1 N}; \ \eta>0\right\}\\
= \Gamma^- \cup \Gamma_0\cup
\Gamma^+.
\end{multline*}

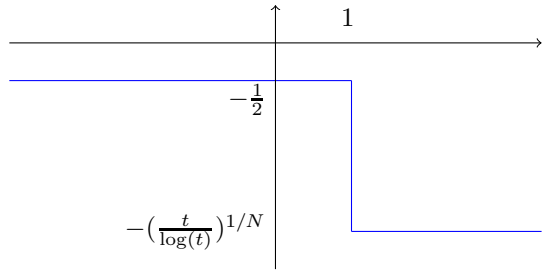
\begin{figure}[ht]
\begin{center}
\begin{tikzpicture}[scale=.5]
\draw[blue](-7, -1) -- ( 2, -1); 
\draw [blue](2, -1) -- (2, -5) ;
\draw[blue] (2,-5) -- (7, -5);
\draw[->] (-7,0) -- (7,0);
\draw[->] (0,-6) -- (0,1);
\draw (-0,-1.5) node[left] {$-\frac 1 2$};
\draw (1.9,.2) node[above] {$1$};
\draw (0,-5) node[left] {$-(\frac{t} {\log(t)} ) ^{1/N} $};
\end{tikzpicture}
\end{center}
\caption{A third contour $\Gamma$.}
\end{figure}
For~$\xi\in \Gamma^-\cup \Gamma_0$, we have for any~$ s\in
\left[\frac 1 3, \frac 2 3\right]$ and any~$ \lambda \in \left[c_1 \bigl(\frac{t}{\log(t)}\bigr) ^{1/ {2N}}, + \infty\right]$
\begin{multline}
\Big\|e^{-is \xi}\frac{V\bigl(s\bigr)}{\bigl(1-i
\xi\bigr)^ k}\frac 1 {\xi +P}\sqrt{\frac{c_0}{2\pi}}e^{-c_0
  \Bigl(\lambda- \xi\bigl(\frac{\log(t)} t \bigr)^{\frac 1 {2N}}\Bigr)^{2}}\Big\|\\
  \leq
\frac C{(1+|\xi|)^k}e^{-\delta(\lambda^2 +\Re(\xi)^2\bigl(\frac{\log(t)} t \bigr)^{\frac 1 {N}}} e^{c_0 \epsilon^2 \bigl(\frac{ t}{\log(t)}  \bigr)^{\frac 1 {N}}}\|U_0\| \ \text{($ \delta>0$)}.
\end{multline}
As a consequence, using again that $k>1$, we deduce that (for $\epsilon>0$ small enough) the contribution of~$\Gamma^-$ to~$ J(0)$ is bounded in norm by
\begin{equation}
\label{5.14bis}
e^{c_0 \epsilon^2 \bigl(\frac{ t}{\log(t)}  \bigr)^{\frac 1 {N}}} \int _{\lambda\geq c_1\log(t)} e^{-\delta
  \lambda ^2}\|U_0\|= {\mathcal {O}}\bigl(e^{-\gamma \bigl(\frac{t}{\log(t)}\bigr) ^{1/ {N}}}\bigr)\|U_0\|, \qquad \gamma >0.
\end{equation}
For~$\xi \in \Gamma^ +$ and $s\in [1/3, 2/3]$ we have
\begin{multline}
\Big\|e^{-is\xi}\frac{V\bigl(s\bigr)}{\bigl(1-i
\xi\bigr)^ k}\frac 1 {\xi +P}\sqrt{\frac{c_0}{\pi}}e^{-c_0
  \bigl(\lambda- \xi\bigl(\frac{\log(t)} t \bigr)^{\frac 1 {2N}}\bigr)^{2}}\Big\|\\
        \leq C e^{(c_0 \epsilon^2-\frac {\epsilon} 3) \bigl(\frac{t}{\log(t)}\bigr) ^{1/ {N}}}\frac
1{(1+|\eta|)^k}e^{-
c_0\bigl(\lambda -(1+ \eta) \bigl(\frac{\log(t)}{t}\bigr) ^{1/ {N}}\bigr)^2}\|U_0\|.
\end{multline}
Hence, using~$k>1$,  we see that the contribution of ~$\Gamma^+$ to~$ J(0)$ is again bounded in norm by
\begin{equation}
\label{5.15bis}
{\mathcal {O}}\left(e^{-\gamma \left(\frac{t}{\log(t)}\right) ^{1/ {N}}}\right)\|U_0\|.
\end{equation}
The contributions to ~$ J\bigl(0\bigr)$ of the region~$ \lambda < -c_1\left(\frac{t}{\log(t)}\right) ^{1/ {N}}$ are bounded similarly.
\par
Finally, it remains to bound
\begin{equation}
\label{5.12bis}
\int_{0}^1 \|K(u)\| du\leq \left(\int_{0}^1 \|K(u)\|^2 du\right)^{1/2}.
\end{equation}
From  Plancherel formula 
\begin{multline}
\label{5.9bis}
\begin{aligned}\int_{-\infty}^{+\infty} \|K(u)\|^2 du&= C
\int_{-\infty}^{+\infty}\Big\| {\frac i {\bigl(1-i
\xi\bigr)^k} \widehat{V\psi '}(\xi)\int_{|\lambda|\geq c_1\bigl(\frac{\log(t)}{t}\bigr) ^{1/ {N}}}e^{-c_0
  \bigl(\lambda- {\xi}\bigl(\frac{\log(t)}{t}\bigr) ^{1/ {2N}}\bigr)^{2}}d\lambda }\Big\|^2 d\xi\\
&= C\int_{-\infty}^{+\infty}\|H(\xi)\|^2 d \xi,
\end{aligned}
\end{multline}
with, for any~$|\xi| >\frac {c_1}{2} \bigl(\frac{t}{\log(t)}\bigr) ^{1/ {N}}$
\begin{equation}
\label{5.10bis}
\begin{aligned}
\|H(\xi)\|& =\left\| \int_{|\lambda|\geq c_1\log(t)}\frac
1{\left(1-i \xi\right)^ k}e^{-c_0
  \left(\lambda- {\xi}\left(\frac{\log(t)}{t}\right) ^{1/ {2N}}\right)^{2}} d\lambda \widehat{V\psi '}\left(\xi\right)
\right\|\\
& \leq  C \left(\frac{\log(t)}{t}\right) ^{k/ {N}}
  \|\widehat{V\psi '}(\xi)\|,
\end{aligned}
\end{equation}
and for~$|\xi| <\frac {c_1}{2} \bigl(\frac{t}{\log(t)}\bigr) ^{1/ {N}}$
 \begin{equation}
\label{5.11bis}
\|H(\xi)\| \leq \int _{|\lambda| > c_1\log(t)}e^{-\delta (\lambda^2 + {\xi}^2\bigl(\frac{\log(t)}{t}\bigr) ^{1/ {N}})}
 d\lambda \|\widehat{V\psi '}(\xi)\|\leq C e^{-\delta\bigl(\frac{t}{\log(t)}\bigr) ^{1/ {N}}}\|\widehat{V\psi '}(\xi)\|. 
\end{equation}
Hence, from~(\ref{5.12bis}),~(\ref{5.9bis})~(\ref{5.10bis}) and~(\ref{5.11bis}), using
\begin{equation}
\int_{-\infty}^{+\infty} \|\widehat {V\psi '}\bigl(\xi\bigr)\|^ 2 d\xi =
\int_{-\infty}^{+\infty} \|\psi ' V\bigl(s\bigr)\|^ 2 ds\leq C \int_0^1 |\psi'
\bigl(s\bigr)|^2 ds \|U_0\|^ 2,
\end{equation}
(let us recall that $V(s) = e^{is B}\chi(x) U_{0}\Rightarrow \| V(s)\|\leq
\|\chi(x) U_{0}\|\leq C\|U_{0}\|$), we deduce that
\begin{equation}
\label{5.18bis}
\int_{0}^1 \|K(u)\| du\leq C \left(\frac{\log(t)} t \right) ^{k/ {N}}\|U_0 \|.
\end{equation}
From~(\ref{5.17bis}),~(\ref{5.7bis})~(\ref{5.14bis}),~(\ref{5.15bis}) and~(\ref{5.18bis}), we get finally 
\begin{equation*}
\left\|I_2\right\|\leq C \left(\frac{\log(t)} t \right) ^{k/ {N}}\|U_0 \|.
\end{equation*}

\subsubsection{Estimate of~$I_2$ for $0<k<1$}\label{4.2.3}

In this section, we are going to recover the case $0<k\leq 1$ from the estimates for $k=0$ and $k=2$ and an interpolation argument.

 This interpolation argument cannot be performed on the final result, but we have to apply it on the operators $I_2$.

 We follow the $k$ dependence of the operators $I_1$ and $I_2$ explicitely and write $I_{1}^k, I_2^k$.

 When $k<1$, $I_2^k$ is only defined by the relation 
$$I_{2}^k (U_0) =  \chi(x)  (1+ iP )^{-k} e^{-itP}\chi(x)U_0 - I_1(U_0),$$
where we observe that the integral defining $I_1^k$ still makes sense.

 According to the estimates in Section~\ref{se.4.1}, we have 
 $$ \| I_{1}^q \|_{\mathcal{L}(L^2)} \leq \frac{C}{ t^{d/2}}, \qquad q=0, k,$$
 which implies that 
 $$ \| I_{2}^0 \|_{\mathcal{L}(L^2)} \leq \| I_{1}^0 \|_{\mathcal{L}(L^2)}+ \|\chi(x)  (1+ iP )^{-k} e^{-itP}\chi(x) \|_{\mathcal{L}(L^2)}\leq C.
 $$
 On the other hand, for $0<k\leq1$, $I_{1}^p$ is clearly the $\frac p 2$-interpolate between
 $I_{1}^0 $ and $I_{1}^2$ and $\chi(x)  (1+ iP )^{-p} e^{-itP}\chi(x)$ is the $\frac p 2$-interpolate between
 $\chi(x)  e^{-itP}\chi(x)$ and $\chi(x)  (1+ iP )^{-2} e^{-itP}\chi(x)$.

  As a consequence, we get 
 \begin{equation}
 \left\|I_2^k \right\|_{\mathcal{L}(L^2)} \leq C \left(\left(\frac{\log(t)} t \right) ^{2/ {N}}\right) ^{\frac p 2}.
 \label{Estiter}
\end{equation}
Collecting estimates (\ref{5.6bis}) and (\ref{Estiter}) gives (\ref{esti}) ends the proof of Theorem \ref{t4}.

\subsection{Proof of Theorem~\ref{t4} (case $- \frac 1 2 \leq N <0$)}
As previously we first focus on the case where assumption (E) holds. We start from a representation formula, for $\epsilon >0$ fixed. If 
$$ U(t) = e^{it P} \frac{ (1+ i P)^k } { (1- i \epsilon P)^{k+2}} \chi(x) U_0, $$
we have
\begin{equation}
\langle x \rangle ^{-d/2} U(t) =\chi(x)\int_{\Ima \xi = 0^-}d\xi
e^{it\xi}\frac{(1+i\xi)^k}{(1- i \epsilon \xi)^{k+2}}\frac 1 {\xi +P}\chi(x) U_0
\end{equation}
 We now deform the  integral above on the  new contour $\Gamma$, noticing that the deformation is licit due to the factor $\frac{(1+i\xi)^k}{(1- i \epsilon \xi)^{k+2}}$ and the boundedness of the resolvent in the domain between the real axis and the contour $\Gamma$. 
\begin{equation}
\Gamma=\left[0,D+\imath cD^{-N}\right]
\cup
\left[0,-D+\imath c D^{-N}\right]
\cup
\left\{ z=\rho+\imath c \rho^{-N}:|\rho|\geq D\right\}
= \Gamma_1 \cup \Gamma_2 \cup \Gamma_3.
\end{equation}
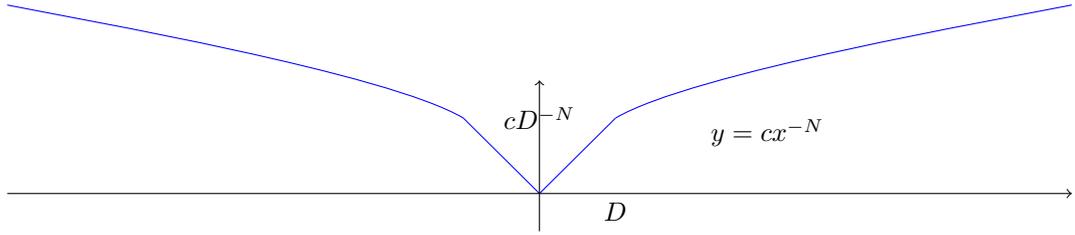
\begin{figure}[ht]

\begin{center}
\begin{tikzpicture}[scale=1]
\draw[blue](0, 0) -- (1, 1); 
\draw [blue](0, 0) -- (-1, 1) ;
\draw[blue] (1,1).. controls (2,1.6) and  (6,2.3) .. (7, 2.5);
\draw[blue] (-1,1).. controls (-2,1.6) and  (-6,2.3) .. (-7, 2.5);
\draw[->] (-7,0) -- (7,0);
\draw[->] (0,-0.5) -- (0,1.5);
\draw (1, 0) node[below] {$D$};
\draw (3,.5) node[above] {$y=cx^{-N}$};
\draw (0,1) node {$cD^{-N}$};
\end{tikzpicture}
\end{center}

\caption{ A fourth contour $\Gamma$ (recall that $- \frac 1 2 <N <0$).}
\label{fig.4}
\end{figure}
Our purpose is to get estimates uniform with respect to $\epsilon >0$ and pass to the limit $\epsilon \rightarrow 0$. 

The contributions of $\Gamma_1$ and $\Gamma_2$ are now estimated exactly as in the previous section, and we only have to
 estimate the contribution of $\Gamma_3$ uniformly with respect to $\epsilon >0$. 

 From~\eqref{e2bis}, we  have for any~$ t>0$ (with a constant $C$ independent on $\epsilon >0$)
\begin{multline}
\left\|\chi(x)\int_{\xi \in \Gamma_3} e^{it\xi}\frac{(1+i\xi)^k}{(1- i \epsilon \xi)^{k+2}}\frac 1 {\xi +P}\langle x \rangle ^{-d/2} U_0\right\|_{\mathcal{H}}\\
\leq  C \sqrt{c_0}\int_{|\eta| \geq D}e ^{-t c|\eta|^{-N}} \frac{ |\eta| ^ k} {1+ \epsilon |\eta|} d\eta \|U_0\| \leq C t^{-d/2} \| U_0\|_{\mathcal{H}}.
\end{multline}

Letting $\epsilon \rightarrow 0$ gives ~\eqref{estibis} under assumption (E).

\end{document}